\documentclass[12pt]{l4dc2021}

\usepackage{amsmath,amssymb,amsfonts,bbm}
\usepackage{algorithmic}
\usepackage{graphicx}
\usepackage{textcomp}
\usepackage{latexsym,enumerate,multicol,wasysym}
\usepackage{color}
\usepackage{nccrules}
\usepackage{xfrac}
\usepackage{hyperref}
\usepackage{tikz}
\usepackage{xcolor}
\usepackage{mathrsfs}
\usepackage{wrapfig}
\usepackage[font=small]{caption}
\def\mf{\mathbf}
\def\mb{\mathbb}
\def\mc{\mathcal}
\def\beq{\begin{equation*}}
\def\eeq{\end{equation*}}
\def\bql{\begin{equation}}
\def\eql{\end{equation}}
\def\bqn{\begin{eqnarray*}}
\def\eqn{\end{eqnarray*}}
\def\bnl{\begin{eqnarray}}
\def\enl{\end{eqnarray}}
\def\bna{\bql\begin{array}{rcl}}
\def\ena{\end{array}\eql}
\def\bnn{\beq\begin{array}{rcl}}
\def\enn{\end{array}\eeq}
\def\bma{\begin{bmatrix}}
\def\ema{\end{bmatrix}}
\def\bmx{\begin{matrix}}
\def\emx{\end{matrix}}
\def\ben{\begin{enumerate}}
\def\een{\end{enumerate}}
\def\bit{\begin{itemize}}
\def\eit{\end{itemize}}
\def\bei{\begin{itemize}}
\def\eei{\end{itemize}}
\def\bet{\begin{tabular}}
\def\eet{\end{tabular}}

\newcommand{\allcaps}[1]{\uppercase\expandafter{#1}}

\def\R{\mb{R}}

\def\Rd{\mb{R}^d}

\allowdisplaybreaks 

\newcommand{\E}{\mathbb{E}}

\title[Control with learning on the fly: System with unknown drift]{Optimal Control with Learning on the Fly: System with Unknown Drift}
\usepackage{times}

\author{\Name {Daniel Gurevich} \Email {dgurevich@princeton.edu}\\
\Name {Debdipta Goswami} \Email {goswamid@princeton.edu}\\
\Name {Charles L. Fefferman} \Email {cf@math.princeton.edu}\\
\Name {Clarence W. Rowley} \Email {cwrowley@princeton.edu}\\
\addr Princeton University, NJ, USA}

\begin{document}

\maketitle

\begin{abstract}%
 This paper derives an optimal control strategy for a simple stochastic
 dynamical system with constant drift and an additive control input. Motivated
 by the example of a physical system with an unexpected change in its dynamics,
 we take the drift parameter to be unknown, so that it must be learned while
 controlling the system. The state of the system is observed through a linear
 observation model with Gaussian noise. In contrast to most previous work, which focuses on a controller's asymptotic performance over an infinite time horizon, we minimize a quadratic cost function over a finite time horizon.
  The performance of our control strategy is quantified by comparing its cost
  with the cost incurred by an optimal controller that has full knowledge of the
  parameters. This approach gives rise to several notions of ``regret.'' We
  derive a set of control strategies that provably minimize the worst-case
  regret; these arise from Bayesian strategies that assume a specific fixed
  prior on the drift parameter. This work suggests that examining Bayesian
  strategies may lead to optimal or near-optimal control strategies for a much
  larger class of realistic dynamical models with unknown parameters.
\end{abstract}


\section{Introduction}
With the proliferation of autonomous controllers in every aspect of human life,
including many safety-critical systems, it is of paramount importance to ensure
their robustness in the face of unexpected changes in the environment and
failures of physical components. Hence, modern controllers must be able to adapt
to unexpected and uncertain changes in a system's dynamics. However, such
capability is currently limited. For instance, while an expert human pilot could
successfully land an aircraft even after damage such as an engine failure, this is an unlikely feat for a present-day autopilot. 

When a physical system undergoes a sudden change in its dynamics, forcing an
adjustment of some model parameters used by the controller, there is only a
limited time and amount of data available to adjust the model. Moreover, in most
safety-critical scenarios, control must be applied without interruption, i.e.,
there is no opportunity for a dedicated learning phase. Recently developed
data-driven learning methods for dynamical systems (for some examples, see
\cite{Brunton2016}, \cite{Cubitt2012}, \cite{Hills2015}, \cite{Schmidt2009}) are either data-intensive or require substantial \emph{a priori} system knowledge. 
\cite{Ahmadi2017} uses the approach of differential inclusions to assess the
safety of a trajectory of an uncontrolled and unknown system; however, it does
not suggest how to identify a control law that will guarantee this safety. A
common approach for controlling an unmodeled system is to first learn the
parameters in the model, and subsequently design an optimal controller from
the learned model. A related approach is to divide the time horizon of the
control into epochs, and after each epoch, first refine the model and then
update the controller accordingly. This latter approach yields asymptotically
optimal results in the limit of large time \citep{Cohen2019}. Other methods such
as the \textit{optimism-in-the-face-of-uncertainty} (OFU) paradigm first
introduced by \cite{lai1985asymptotically} use simultaneous estimation of the
system parameter and control design \cite{abbasi2011regret} that theoretically
achieves a $\Tilde{O}(\sqrt T)$ regret bound in the large time limit.  However,
none of the aforementioned methods are well-suited to control of a
safety-critical system on the most relevant short time scales.

The present paper specifically deals with the problem of finding an optimal control on a finite time horizon with no distinct learning and control phases. In our model, we consider a stochastic dynamical system with an unknown constant drift, an additive control input, and a Wiener process noise. The system state is observed through a linear measurement that is corrupted by an independent Wiener measurement noise. Thus, both the process and observation dynamics obey linear stochastic differential equations (SDEs).

This paper builds on the prior work of \cite{Fefferman2020}, which treated a
similar problem with multiple special restrictions. In that work, it was assumed
that the system state is one-dimensional and can be measured without any sensor
noise. The present paper considers several generalizations: the state is
multi-dimensional, correlations between different components of the state are
taken into account, and measurements are noisy. 

The paper is organized as follows. In Section~\ref{Sec: Prob_form} we formulate
the Bayesian optimal control problem, in which we have a prior probability distribution for the unknown parameter, as well as the ``agnostic'' optimal control problem, for which we have no prior belief, and optimality is based on the notion of ``regret.''  In Section~\ref{Sec: Bayesian_Strategy} we show that Bayesian
strategies are candidate optimal strategies for agnostic control.
In Section~\ref{Sec: Agnostic} we present agnostic control strategies that
minimize the worst-case multiplicative or additive regret. In Section~\ref{Sec: Conclusions} we summarize our findings and discuss future research
directions. Proofs for the theorems in Sections \ref{Section: Constant_Regret_Theorem} and \ref{Sec: Agnostic} are provided in the appendices.

\section{Problem Formulation}\label{Sec: Prob_form}
In the discrete version of our problem, we consider a dynamical system with noisy observations on the time interval $[0,T]$
\bql\label{Eq: Discrete_Dynamics}
\Delta\mf{q} = (\mf{a} + \mf{u}(t))\Delta t + \Delta\mf{W}, \qquad
\Delta\mf{y} = \mf{q}\Delta t + \Delta \mf{V},
\eql 
where $\mf{q}\in\Rd$ denotes the position of a particle which is observed through a noisy measurement $\mf{y}\in\Rd$. $\Delta\mf{W}$ and $\Delta\mf{V}$ are two independent normally distributed random variables with zero mean and standard deviation $\Sigma_W\sqrt{\Delta t}$ and $\Sigma_V\sqrt{\Delta t}$ respectively. A simple scaling and rotation allow us to take $\Sigma_W = \mf{I}_{d\times d}$, as we will do from now on. We also take the initial conditions $\mf{y}(0) = 0$ and $\mf{q}(0)$ normally distributed with $\E[\mf{q}(0)\mf{q}(0)^T] = \Sigma_{\mf{ q}_0}$. As $\Delta t \rightarrow 0$, we obtain the continuous-time version of the dynamics, 
\bql\label{Eq: Continuous_Dynamics}
d\mf{q}(t) = (\mf{a} + \mf{u}(t))dt + d\mf{W}(t)\qquad
d\mf{y}(t) = \mf{q}(t)dt + d\mf{V}(t),
\eql 
where $\mf{W}(t)$ and $\mf{V}(t)$ are independent $d$-dimensional scaled Wiener
processes with their generalized derivatives $d\mf{W}(t)$ and $d\mf{V}(t)$
corresponding to white noise. Given some $T_0 \in [0, T)$, we are allowed to observe the system for $t\in[0,
T_0]$ without applying any control, after which we control the system on $t \in
[T_0, T]$. The main focus of this paper is on the most practically challenging case arising when $T_0 = 0$, i.e., there is no dedicated learning phase and the control must be applied from the beginning. 
Our objective is to find the optimal control strategy $\mf{u}(t)$, depending on the measurements $\{y(\tau) : 0 \leq \tau < t\}$ and taking values in $\Rd$,
that minimizes the cost function
\bql \label{Eq: Cost}
J(\mf{q},\mf{u};\mf{a}) = \E\left[\;\int\limits_{T_0}^T \left( \mf{q}^T Q\mf{q} + \mf{u}^T R\mf{u}\right)dt\right].
\eql
If the drift parameter $\mf{a}$ is known, then the expectation in \eqref{Eq:
  Cost} is well-defined, and the task reduces to a classical optimal control
problem. However, if $\mf{a}$ is unknown, it is not immediately clear what should be considered an optimal control strategy. 

\subsection{Notions of Optimality}
When the parameter $\mf{a}$ is unknown, the performance of a control strategy can be quantified in two different ways.
\subsubsection{Bayesian Control}\label{Bayesian}
In the Bayesian version of our problem, a prior belief is assumed on $\mf{a}$
--- that is, $\mf{a}$ is chosen at random according to a probability measure $\mu(\mf{a})$. The expected cost associated with $\mf{u}$ can then be computed as 
\bql \label{Eq: Cost_Bayesian}
\mc{J}(\mf{q},\mf{u};\mu) = \E_{\mf{a}}\left[J(\mf{q},\mf{u};\mf{a})\right] = \int_{\Rd} J(\mf{q},\mf{u};\mf{a})\,d\mu(\mf{a}).
\eql
In this case, the objective is to find a control strategy $\mf{u}(t)$
that minimizes the expected value in \eqref{Eq: Cost_Bayesian} given the prior belief $\mu(a)$.

\subsubsection{Agnostic Control}
\begin{wrapfigure}[12]{l}{0.8\textwidth}
\centering
\includegraphics[trim=0cm 1cm 0cm 1cm, clip=true, width=0.8\textwidth]{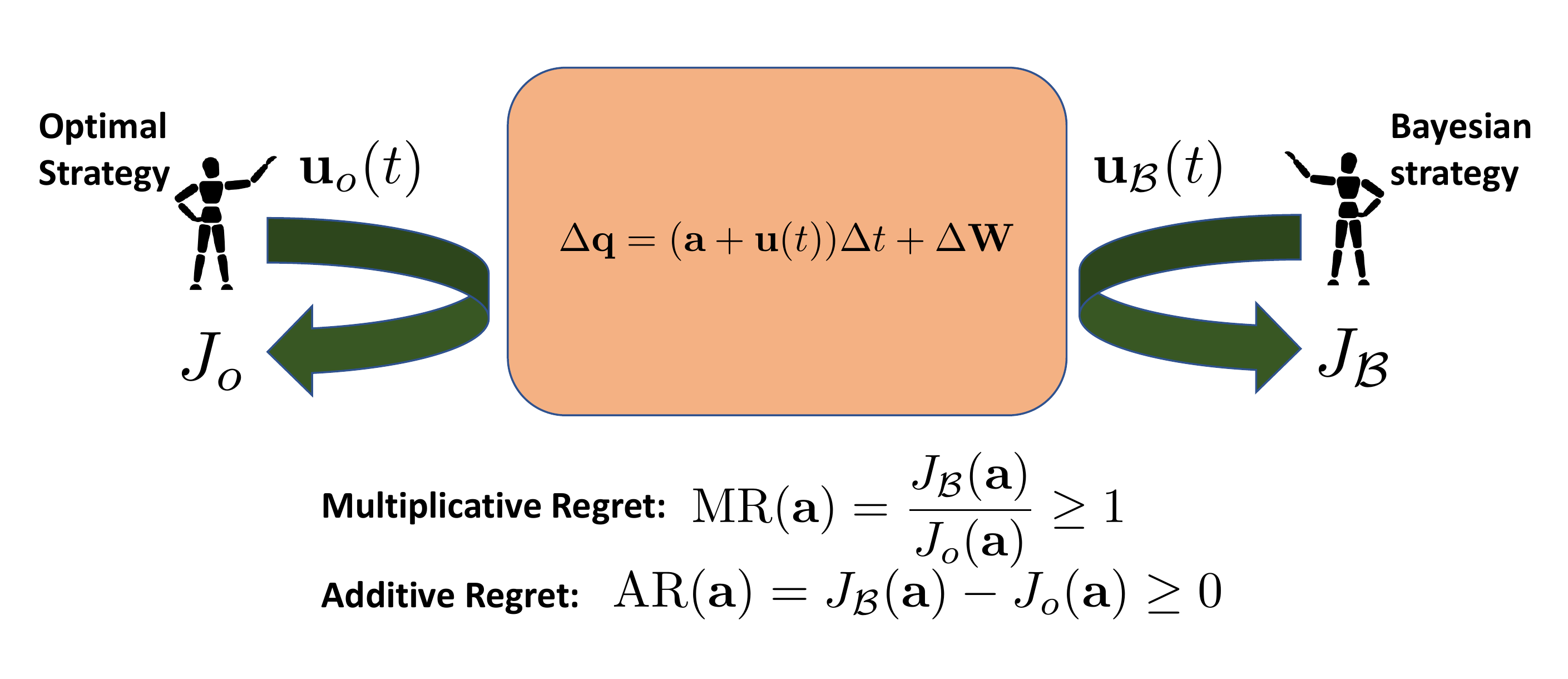}
\caption{Different notions of ``regret''.\label{Fig: Regret}}
\end{wrapfigure}

While Bayesian strategies yield a well-defined cost and performance metric, their inherent dependence on the prior belief makes them unsuitable when no prior information on the parameter is available. We wish to find a strategy that does not depend on any prior belief on $\mf{a}$. We call such a strategy an \emph{agnostic strategy}, and we use the notion of \emph{regret} to measure its performance, in particular comparing against the optimal strategy when $\mf{a}$ is known. Suppose the agnostic strategy $\mc{B}$ yields a cost $J_{\mc{B}}(\mf{a})$ according to \eqref{Eq: Cost}, which depends on the value of $\mf{a}$. The optimal strategy with full knowledge of $\mf{a}$ instead incurs a cost $J_o(\mf{a})$. Clearly, $J_o(\mf{a})\leq J_{\mc{B}}(\mf{a})$ for all $\mc{B}$ and $\mf{a}$. The regret achieved by the agnostic strategy $\mc{B}$ can be defined in various ways. In all cases, we seek a strategy that minimizes the worst-case regret across all possible values of $\mf{a}$. In this paper, we focus on the following definitions of regret:
\begin{itemize}
    \item \emph{Additive regret}: The additive regret (often called the regret in other literature) is the difference
    \beq \operatorname{AR}_{\mc{B}}(\mf{a}) \triangleq J_{\mc{B}}(\mf{a}) - J_o(\mf{a}) \geq 0. \eeq
    A strategy that minimizes the worst-case additive regret $\operatorname{AR}^*_{\mc{B}} = \sup_{\mf{a}}\operatorname{AR}_{\mc{B}}(\mf{a})$ is deemed optimal.
    \item \emph{Multiplicative regret}: The multiplicative regret or \emph{competitive ratio}, as it is more commonly called in the literature, is the ratio
    \beq \operatorname{MR}_{\mc{B}}(\mf{a}) \triangleq \dfrac{J_{\mc{B}}(\mf{a})}{J_o(\mf{a})} \geq 1. \eeq 
    Under this definition, the optimal agnostic strategy minimizes the worst-case multiplicative regret $\operatorname{MR}^*_{\mc{B}} = \sup_{\mf{a}}\operatorname{MR}_{\mc{B}}(\mf{a})$. Fig.~\ref{Fig: Regret} shows a schematic for different notions of regret.
\end{itemize}

\subsection{Bayesian Strategies Are Candidate Agnostic Strategies} \label{Section: Constant_Regret_Theorem}
While a Bayesian strategy can generally be identified by variational methods, a
good agnostic strategy cannot readily be constructed: the worst-case regret of a
strategy is a complex nonlinear function of both the optimal cost achievable at
fixed $\mf{a}$ and the strategy's cost, examined over all possible values of
$\mf{a}$. However, we may search for optimal or near-optimal agnostic strategies
which arise as Bayesian strategies from a fixed prior. We are assisted by the
following theorem, which provides both lower and upper bounds on the optimal
worst-case regret.
%
\begin{theorem}\label{Thm: Regret_Bounds}
Suppose that $\mc{A}$ is the Bayesian strategy for the prior $\mu$ and $\mc{B}$
is an optimal agnostic strategy in the additive or multiplicative regret setting. Then respectively
\begin{equation} \label{Eq: AR_Bound}
    \int \operatorname{AR}_{\mc{A}}(\mf{a})\,d\mu(\mf{a}) \leq \operatorname{AR}^*_{\mc{B}} \leq \operatorname{AR}^*_{\mc{A}}
\end{equation}
or
\begin{equation} \label{Eq: MR_Bound}
    \frac{\int \operatorname{MR}_{\mc{A}}(\mf{a})J_o(\mf{a})\,d\mu(\mf{a})}{\int J_o(\mf{a})\,d\mu(\mf{a})} \leq \operatorname{MR}^*_{\mc{B}} \leq \operatorname{MR}^*_{\mc{A}}.
\end{equation}
\end{theorem}

\begin{proof}
The proof is given in Appendix \ref{App: Regret_Bounds}.
\end{proof}
As a corollary, we obtain a simple criterion for showing that a Bayesian strategy is an optimal agnostic strategy, which was previously used in \cite{Fefferman2020}.
\begin{corollary} \label{Cor: Constant_Regret}
Suppose the Bayesian strategy $\mc{A}$ is optimal under the prior $\mu$, achieves constant regret $R$ (additive or multiplicative) for all $\mf{a}$ in the support of $\mu$, and achieves regret not exceeding $R$ for all $\mf{a}$ not in the support of $\mu$. Then $\mc{A}$ is an optimal agnostic strategy which achieves a worst-case regret of $R$.
\end{corollary}
\begin{proof}
Under the conditions of this corollary, both the lower and upper bounds on the optimal worst-case regret in Theorem \ref{Thm: Regret_Bounds} are equal to $R$, which concludes the proof.
\end{proof}
In particular, any Bayesian strategy achieving constant regret for all $\mf{a}$ is optimal by the corollary. Thus, the plan of attack in this paper will be as follows. First, we construct the optimal Bayesian strategies for a specific set of priors $\mu$. Then we will compute the regret for those strategies for the values of $\mf{a}$ and solve for the specific prior $\mu$ that provides us with an optimal Bayesian strategy with constant regret, which is also an optimal agnostic strategy.  

\section{A Bayesian Strategy for Unknown $\mf{a}$}\label{Sec: Bayesian_Strategy}
This section demonstrates the optimal strategy for the Bayesian problem discussed in Section \ref{Bayesian} given a fixed prior belief $\mu(\mf{a})$. We restrict ourselves to Gaussian prior beliefs only, i.e., 
$ d\mu(\mf{a}) = \rho(\mf{a})\,d\mf{a},$ where $\rho(\mf{a}) = \mc{N}(\mf{a};0, \Sigma)$ for some covariance matrix $\Sigma \succcurlyeq 0$. This offers the advantage that the posterior distributions of $\mf{a}$ and $\mf{q}$ at any time $t$ will remain jointly Gaussian. To approach this problem in a more traditional control-theoretic way, we write the system \eqref{Eq: Discrete_Dynamics} as follows:
\bql\label{Eq: Modified_Dynamics}
d\mf{x} = (F\mf{x} + B\mf{u})\,dt + G\,d\mf{W} \qquad
d\mf{y} = H\mf{x} \,dt + d\mf{V},
\eql
where $\mf{x} = \begin{bmatrix}\mf{q} \\ \mf{a} \end{bmatrix} \in \R^{2d}$, $F = \begin{bmatrix} \mf{0}_{d\times d} & \mf{I}_{d\times d} \\ \mf{0}_{d\times d} & \mf{0}_{d\times d} \end{bmatrix}$, $G = B = \begin{bmatrix} \mf{I}_{d\times d} \\ \mf{0}_{d\times d} \end{bmatrix}$, and $H = \begin{bmatrix} \mf{I}_{d\times d} & \mf{0}_{d\times d} \end{bmatrix}$. 
The cost \eqref{Eq: Cost_Bayesian} can equivalently be written as 
\bql \label{Eq: Cost_modified} \mc{J}(\mf{x},\mf{u}) = \E\left[\;
  \int\limits_{T_0}^T (\mf{x}^T \tilde{Q} \mf{x} + \mf{u}^T R \mf{u})\,dt
\right],\eql
with $\tilde{Q} = \begin{bmatrix} Q & 0_{d\times d}\\ 0_{d\times d} & 0_{d\times
    d} \end{bmatrix}$. Now the problem of minimizing \eqref{Eq: Cost_modified}
subject to the dynamics \eqref{Eq: Modified_Dynamics} becomes a standard
linear-quadratic-Gaussian optimal control problem (see, for instance, \cite{SpeyerChap9}).
Let $\hat{\mf{x}}(t) = \begin{bmatrix} \hat{\mf{q}}(t) \\ \hat{\mf{a}}(t)\end{bmatrix} = \E \left[ \mf{x}(t) | \mf{x}(0), \mf{y}(\tau): \tau\in[0,t)\right]$ be the Bayesian estimates of $\mf{q}(t)$ and $\mf{a}$ at time $t$. Then $\hat{\mf{q}}(t)$ and $\hat{\mf{a}}(t)$ satisfy the continuous-time Kalman filter equations
\begin{equation}\label{Eq: Estimation_ODE}
\begin{aligned}
d\hat{\mf{q}}(t) &= \left(\hat{\mf{a}} + \mf{u}\right)dt + P_{11}(t)\Sigma_V^{-1}\left(d\mf{y} - \hat{\mf{q}}dt\right)\\
d\hat{\mf{a}}(t) &= P_{12}^T(t)\Sigma_V^{-1}\left(d\mf{y} - \hat{\mf{q}}\,dt\right)
\end{aligned}
\end{equation}
%
%
with $P(t) = \begin{bmatrix} P_{11} & P_{12}\\ P_{12}^T & P_{22} \end{bmatrix} = \E[(\mf{x}(t) - \hat{\mf{x}}(t))(\mf{x}(t) - \hat{\mf{x}}(t))^T]$ given by the Riccati ODEs
\begin{align}\label{Eq: Estimation_Riccati}
\dot{P}(t) = FP + PF^T -PH^T\Sigma_V^{-1}HP + G\Sigma_WG^T,\quad P(0) = \begin{bmatrix} \Sigma_{\mf{q}_0} & 0 \\0 & \Sigma \end{bmatrix}.
\end{align}
The optimal cost to go at time $t$ is 
\begin{equation}
\begin{aligned}
\mc{J}(t, \mf{q}, \mf{y}; \mu)
&= \E \left[\int\limits_t^T(\mf{x}^T\tilde{Q}\mf{x} + \mf{u}^TR\mf{u})d\tau\right]\\
&= \E \left[\int\limits_t^T(\hat{\mf{q}}^TQ\hat{\mf{q}} + \mf{u}^TR\mf{u})d\tau\right] + \int\limits_t^T\operatorname{Trace}(P(\tau)\tilde{Q})d\tau.
\end{aligned}
\end{equation}
The optimal value of $\mc{J}$ solves the Hamilton-Jacobi-Bellman equation and takes the form
\bql
\mc{J}(t) = \hat{\mf{x}}^TS(t)\hat{\mf{x}} + \alpha(t) + \int\limits_t^T\operatorname{Trace}(P(\tau)\tilde{Q})d\tau.
\eql

In turn, the optimal control is given by
\bql \label{Eq: Optimal_Control}
\mf{u}^*(t) = R^{-1}B^TS(t)\hat{\mf{x}}(t) = -R^{-1}(S_{11}\hat{\mf{q}}(t) + S_{12}\hat{\mf{a}}(t)).
\eql

The matrix $S = \begin{bmatrix} S_{11} & S_{12}\\ S_{12}^T & S_{22} \end{bmatrix}$ and scalar-valued function $\alpha$ satisfy the control Riccati equations
\begin{equation}\label{Eq: Control_Riccati}
\begin{aligned}
-\dot{S} &= SF + F^TS + \tilde{Q} - SBR^{-1}B^TS, \quad S(T) = 0,\\
-\Dot{\alpha} &= \operatorname{Trace}\left(PH^T\Sigma_V^{-1}HPS\right),\quad \alpha(T) = 0.
\end{aligned}
\end{equation}
After some algebraic manipulation, the optimal cost-to-go at time $t$ becomes
\begin{equation}
\begin{aligned}
    \mc{J}(t, \mf{q}, \mf{y}; \mu) = (\hat{\mf{q}}^TS_{11}\hat{\mf{q}} + \hat{\mf{a}}^TS_{22}\hat{\mf{a}} + 2\hat{\mf{q}}^TS_{12}\hat{\mf{a}}) + \operatorname{Trace}(P(t)S(t)) \\+ \operatorname{Trace}\Bigg[\int\limits_t^T G\Sigma_WG^TS(t) + S(t)BR^{-1}B^TP(t) d\tau\Bigg].
\end{aligned}
\end{equation}
\section{Towards Agnostic Control: Performance of the Bayesian Strategy} \label{Sec: Agnostic}
 In Section \ref{Section: Constant_Regret_Theorem}, it is established that a
Bayesian strategy that has constant regret independent of ${\mf a}$ will minimize worst-case regret. To obtain such a strategy, the performance of a general Bayesian strategy needs to be quantified for each fixed true value of the parameter ${\bf a}$. The total cost incurred at ${\bf a}$ by the Bayesian strategy arising from a prior $\mu$ is 
\bql \label{Eq: Cost_strategy}
\mc{J}(\mf{q}, \mf{a}; \mu) = \E\left[\;\int\limits_{T_0}^T \left( \mf{q}^T Q\mf{q} + \mf{u}^T R\mf{u}\right)dt\right],
\eql
where $\mf{u} = R^{-1}B^tS\hat{\mf{x}} = R^{-1} (S_{11}\hat{\mf{q}} + S_{12}\hat{\mf{a}})$. With this control, the stochastic differential equations describing $\mf{q}$, $\hat{\mf{q}}$, and $\mf{a}$ become
\begin{align}\label{Eq: SDE_Strategy}
d\mf{q} &= \left[\mf{a} - R^{-1} (S_{11}\hat{\mf{q}} + S_{12}\hat{\mf{a}})\right] dt + d\mf{W},\\\nonumber
d\hat{\mf{q}} &= \left[ \hat{\mf{a}} - R^{-1} (S_{11}\hat{\mf{q}} + S_{12}\hat{\mf{a}})\right]dt + P_{11}\Sigma_V^{-1}\left[(\mf{q} - \hat{\mf{q}})dt +d\mf{V}\right],\\\nonumber
d\hat{\mf{a}} &= P_{12}\Sigma_V^{-1}\left[(\mf{q} - \hat{\mf{q}})dt +d\mf{V}\right]
\end{align}

Let $\Bar{\mf{a}}(t) \triangleq \E[\hat{\mf{a}}(t)]$, $\Bar{\mf{q}}(t) \triangleq \E[\hat{\mf{q}}(t)]$, and $\check{\mf{q}}(t) \triangleq \E[\mf{q}(t)]$ with the associated covariances $\Sigma_{\hat{\mf{a}}\hat{\mf{a}}}$, $\Sigma_{\hat{\mf{q}}\hat{\mf{q}}}$, $\Sigma_{\hat{\mf{q}}\hat{\mf{a}}}$, $\Sigma_{\mf{q}\mf{q}}$, $\Sigma_{\mf{q}\hat{\mf{q}}}$, and $\Sigma_{\mf{q}\hat{\mf{a}}}$. The cost \eqref{Eq: Cost_strategy} can be written as 

\begin{equation} \label{Eq: Cost_strategy_breakdown}
\begin{aligned}
\mc{J}(\mf{q}, \mf{a}; \mu)
&= \E\left[\;\int\limits_{T_0}^T \left( \mf{q}^T Q\mf{q} + \mf{u}^T R\mf{u}\right)dt\right]
= \int\limits_{T_0}^T\left[ \check{\mf{q}}^TQ\check{\mf{q}} + (S_{11}\bar{\mf{q}} + S_{12}\bar{\mf{a}})^TR^{-1}(S_{11}\bar{\mf{q}} + S_{12}\bar{\mf{a}})\right]dt \\
& + \int\limits_{T_0}^T\operatorname{Trace}\bigg(Q\Sigma_{\mf{q}\hat{\mf{q}}} + S_{11}^TR^{-1}S_{11}\Sigma_{\hat{\mf{q}}\hat{\mf{q}}} 
 + 2S_{11}^TR^{-1}S_{12}\Sigma_{\hat{\mf{q}}\hat{\mf{a}}} + S_{12}^TR^{-1}S_{12}\Sigma_{\hat{\mf{a}}\hat{\mf{a}}}  \bigg)dt.
\end{aligned}
\end{equation}

\begin{theorem}\label{Thm: Cost_of_a_strategy}
The cost incurred by a specific Bayesian strategy can be expressed as 
\bql \mc{J}(\mf{q}, \mf{a}; \mu) = \mf{a}^TX(\Sigma, \Sigma_V, T_0, T) \mf{a} + Y(\Sigma, \Sigma_V, T_0, T), \text{ where}
\eql
\beq
\begin{aligned}
X(\Sigma, \Sigma_V, T_0, T) &= \int\limits_{T_0}^T\bigg[ C_1^TQC_1 + (S_{11}C_2 + S_{12}C_3)^TR^{-1}(S_{11}C_2 + S_{12}C_3)\bigg]dt\text{ and}\\
Y(\Sigma, \Sigma_V, T_0, T) &= \int\limits_{T_0}^T\operatorname{Trace}\bigg(Q\Sigma_{\mf{q}\mf{q}} + S_{11}^TR^{-1}S_{11}\Sigma_{\hat{\mf{q}}\hat{\mf{q}}} + 2S_{11}^TR^{-1}S_{12}\Sigma_{\hat{\mf{q}}\hat{\mf{a}}} + S_{12}^TR^{-1}S_{12}\Sigma_{\hat{\mf{a}}\hat{\mf{a}}}  \bigg)dt
\end{aligned}
\eeq
for time-varying functions $C_1(t), C_2(t),$ and $C_3(t)$. All of the quantities appearing here can be found by solving the following system of ODEs:
\begin{equation}\label{Eq: ODE_Strategy}
\begin{aligned}
\dot{C_1} &= I - R^{-1}(S_{11}C_2 - S_{12}C_3)\\
\dot{C_2} &= C_3 - R^{-1}(S_{11}C_2 - S_{12}C_3) + P_{11}\Sigma_V^{-1}(C_1-C_2)\\
\dot{C_3} &= P_{12}\Sigma_V^{-1}(C_1 - C_2)\\
\dot{\Sigma}_{\hat{\mf{a}}\hat{\mf{a}}} &= (\Sigma_{\mf{q}\hat{\mf{a}}} - \Sigma_{\hat{\mf{q}}\hat{\mf{a}}})^T\Sigma_V^{-1}P_{12} + P_{12}\Sigma_V^{-1}(\Sigma_{\mf{q}\hat{\mf{a}}} - \Sigma_{\hat{\mf{q}}\hat{\mf{a}}}) + P_{12}\Sigma_V^{-1}P_{12} \\
\dot{\Sigma}_{\mf{q}\mf{q}} &= \Sigma_W - (\Sigma_{\mf{q}\hat{\mf{q}}}S_{11} + \Sigma_{\mf{q}\hat{\mf{a}}}S_{12}^T)R^{-1} - R^{-1}(\Sigma_{\mf{q}\hat{\mf{q}}}S_{11} + \Sigma_{\mf{q}\hat{\mf{a}}}S_{12}^T)^T\\
\dot{\Sigma}_{\mf{q}\hat{\mf{q}}}&= (\Sigma_{\mf{q}\mf{q}} - \Sigma_{\mf{q}\hat{\mf{q}}})\Sigma_V^{-1}P_{11} - \Sigma_{\mf{q}\hat{\mf{q}}}S_{11}^TR^{-1} +\Sigma_{\mf{q}\hat{\mf{a}}}(I - R^{-1}S_{12})^T - R^{-1}(S_{11}\Sigma_{\hat{\mf{q}}\hat{\mf{q}}} + S_{12}\Sigma_{\hat{\mf{q}}\hat{\mf{a}}}^T)\\
\dot{\Sigma}_{\mf{q}\hat{\mf{a}}} &= (\Sigma_{\mf{q}\mf{q}} - \Sigma_{\mf{q}\hat{\mf{q}}})\Sigma_V^{-1}P_{12} - R^{-1}(S_{11}\Sigma_{\hat{\mf{q}}\hat{\mf{a}}} + S_{12}\Sigma_{\hat{\mf{a}}\hat{\mf{a}}}^T)\\
\dot{\Sigma}_{\hat{\mf{q}}\hat{\mf{q}}}&= P_{11}^T\Sigma_V^{-1}P_{11} + \Sigma_{\hat{\mf{q}}\hat{\mf{a}}}(I -R^{-1}S_{12})^T + (I -R^{-1}S_{12})\Sigma_{\hat{\mf{q}}\hat{\mf{a}}}^T + (\Sigma_{\mf{q}\hat{\mf{q}}} - \Sigma_{\hat{\mf{q}}\hat{\mf{q}}})^T\Sigma_V^{-1}P_{11}\\ & + P_{11}\Sigma_V^{-1}(\Sigma_{\mf{q}\hat{\mf{q}}} - \Sigma_{\hat{\mf{q}}\hat{\mf{q}}}) - \Sigma_{\hat{\mf{q}}\hat{\mf{q}}}S_{11}R^{-1} - R^{-1}S_{11}\Sigma_{\hat{\mf{q}}\hat{\mf{q}}}\\
\dot{\Sigma}_{\hat{\mf{q}}\hat{\mf{a}}}&= (\Sigma_{\mf{q}\hat{\mf{q}}} - \Sigma_{\hat{\mf{q}}\hat{\mf{q}}})^T\Sigma_V^{-1}P_{12} + (I - R^{-1}S_{12})\Sigma_{\hat{\mf{a}}\hat{\mf{a}}}  - R^{-1}S_{12}\Sigma_{\hat{\mf{q}}\hat{\mf{a}}} + P_{11}\Sigma_V^{-1}(\Sigma_{\mf{q}\hat{\mf{a}}} - \Sigma_{\hat{\mf{q}}\hat{\mf{a}}} + P_{12})\\
\dot{\Sigma}_{\hat{\mf{a}}\hat{\mf{a}}}&= (\Sigma_{\mf{q}\hat{\mf{a}}} - \Sigma_{\hat{\mf{q}}\hat{\mf{a}}})^T\Sigma_V^{-1}P_{12} + P_{12}\Sigma_V^{-1}(\Sigma_{\mf{q}\hat{\mf{a}}} - \Sigma_{\hat{\mf{q}}\hat{\mf{a}}}) + P_{12}\Sigma_V^{-1}P_{12}.
\end{aligned}
\end{equation}
\end{theorem}
\begin{proof}
The proof is given in Appendix \ref{App: Cost_of_a_strategy}.
\end{proof}

\subsection{Optimal Cost for a Bayesian Strategy with Known $\mf{a}$}
If $\mf{a}$ is known, then $\hat{\mf{a}}(t) \equiv \bar{\mf{a}}(t) \equiv \mf{a}$, $\Sigma_{\hat{\mf{a}}\hat{\mf{a}}}(t)\equiv 0$, $\Sigma_{\hat{\mf{q}}\hat{\mf{a}}}(t)\equiv 0$, and $\Sigma_{{\mf{q}}\hat{\mf{a}}}(t)\equiv 0$. We use similar notation as above, denoting all of the analogous quantities for this case with a superscript $*$. Hence the Riccati equation \eqref{Eq: Estimation_Riccati} becomes 
\bql \label{Eq: Estimation_Riccati_known_a}
\Dot{P}^*_{11} = - P^{*^T}_{11}\Sigma_V^{-1}P^*_{11} + \Sigma_W,\,P^*_{11}(0) = \Sigma_{\mf{q}_0},
\eql 
and $P^*_{12}(t)=P^*_{22}(t)\equiv 0$. The estimation equation \eqref{Eq: Estimation_ODE} becomes
\bql\label{Eq: Estimation_ODE_known_a}
d\hat{\mf{q}}^*(t) = \left(\mf{a} + \mf{u}\right) dt + P^*_{11}(t)\Sigma_V^{-1}\left(d\mf{y} - \hat{\mf{q}}^*dt\right).
\eql
This, in turn, yields the optimal control $\mf{u}^*(t) = -R^{-1}(S_{11}\hat{\mf{q}}^*(t) + S_{12}{\mf{a}}(t))$.
\begin{theorem} \label{Thm: Cost_for_known_a}
The optimal cost for the Bayesian strategy with known $\mf{a}$ can be expressed as 
\bql\label{Eq: Cost_known_a} J(\mf{q}, \mf{a}) =  \mf{a}^TX^*(\Sigma_V, T_0, T) \mf{a} + Y^*(\Sigma_V, T_0, T),\eql 
\begin{equation*}
\begin{aligned}
\text{where} \quad X^*(\Sigma_V, T_0, T) &= \int\limits_{T_0}^T\bigg[C_1^{*^T}QC^*_1 + (S_{11}C^*_2 + S_{12})^TR^{-1}(S_{11}C^*_2 + S_{12})\bigg]dt\\
\text{and}\quad Y^*(\Sigma_V, T_0, T) &= \int\limits_{T_0}^T\operatorname{Trace}\bigg(Q\Sigma_{\mf{q}\mf{q}}^* + S_{11}^TR^{-1}S_{11}\Sigma_{\hat{\mf{q}}\hat{\mf{q}}}^*\bigg)dt,
\end{aligned}
\end{equation*}
with the time-varying functions $C_1^*(t), C_2^*(t), \Sigma_{\mf{qq}}^*$, and $\Sigma_{\mf{\hat{q}\hat{q}}}^*$ described by the ODEs
\begin{equation}\label{Eq: ODE_Strategy_Known_a}
\begin{aligned}
\dot{C_1}^* &= I - R^{-1}(S_{11}{C_2}^* - S_{12})\\
\dot{C_2}^* &= I - R^{-1}(S_{11}{C_2}^* - S_{12}) + P_{11}\Sigma_V^{-1}(C_1^* - C_2^*)\\
\dot{\Sigma}_{\mf{q}\mf{q}}^* &= \Sigma_W - (\Sigma_{\mf{q}\hat{\mf{q}}}^*S_{11}R^{-1} + R^{-1}S_{11}\Sigma_{\mf{q}\hat{\mf{q}}}^{*^T})\\
\dot{\Sigma}_{\mf{q}\hat{\mf{q}}}^*&= (\Sigma_{\mf{q}\mf{q}}^* - \Sigma_{\mf{q}\hat{\mf{q}}}^*)\Sigma_V^{-1}P_{11} - \Sigma_{\mf{q}\hat{\mf{q}}}^*S_{11}^TR^{-1} - R^{-1}S_{11}\Sigma_{\hat{\mf{q}}\hat{\mf{q}}}^* \\
\dot{\Sigma}_{\hat{\mf{q}}\hat{\mf{q}}}^*&= P_{11}^{*^T}\Sigma_V^{-1}P_{11} + (\Sigma_{\mf{q}\hat{\mf{q}}}^* - \Sigma_{\hat{\mf{q}}\hat{\mf{q}}}^*)^T\Sigma_V^{-1}P_{11}^* + P_{11}^*\Sigma_V^{-1}(\Sigma_{\mf{q}\hat{\mf{q}}}^* - \Sigma_{\hat{\mf{q}}\hat{\mf{q}}}^*) \\ &- \Sigma_{\hat{\mf{q}}\hat{\mf{q}}}^*S_{11}R^{-1} - R^{-1}S_{11}\Sigma_{\hat{\mf{q}}\hat{\mf{q}}}^*.
\end{aligned}
\end{equation}
\end{theorem}
\begin{proof}
The proof is given in Appendix \ref{App: Cost_for_known_a}.
\end{proof}

\subsection{Minimizing Multiplicative Regret} \label{Sec: Mult_Regret}
\begin{figure*}[t]
\subfigure{%
\includegraphics[trim=0cm 0cm 0cm 0cm, clip=true, width=0.5\textwidth]{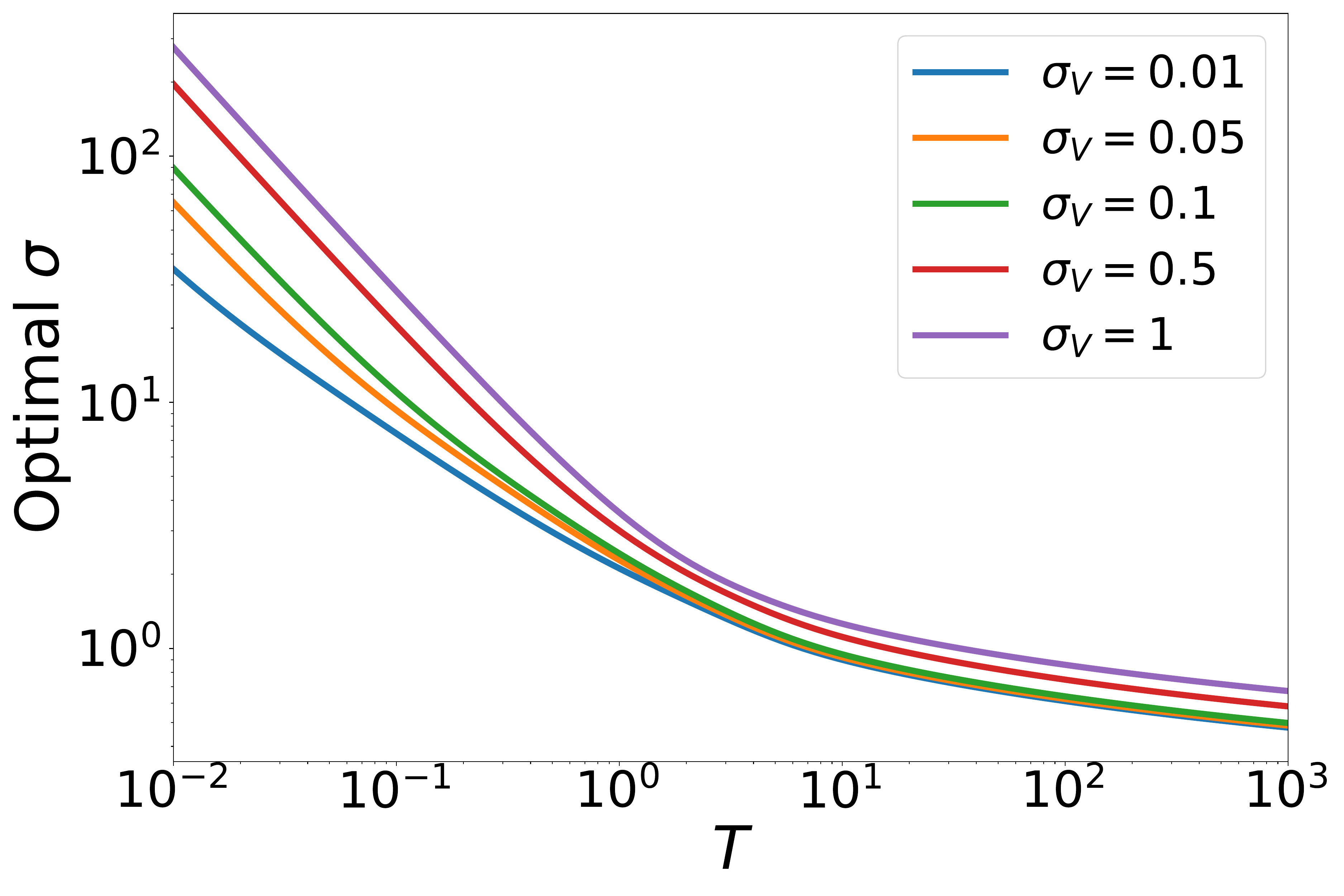}
}
\subfigure{%
\includegraphics[trim=0cm 0cm 0cm 0cm, clip=true, width=0.5\textwidth]{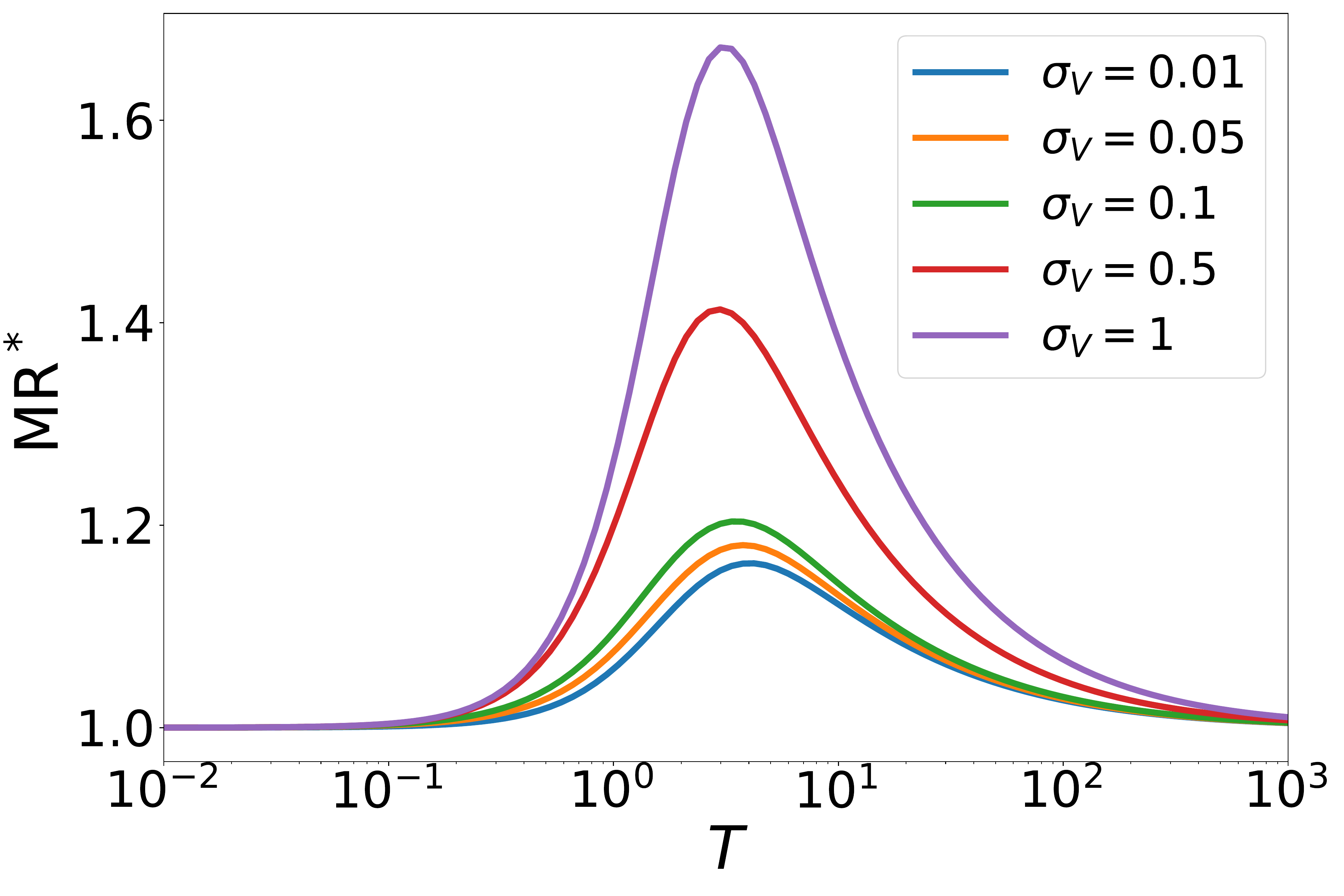}
}
\caption{Bayesian strategies minimizing the worst-case multiplicative regret with varying standard deviation $\sigma_V$ of the sensor noise. (a) Optimal standard deviation $\sigma$ of the prior; (b) worst-case multiplicative regret $\operatorname{MR}^*$. \label{Fig: MultiplicativeRegret}}
\end{figure*}

The optimal cost \eqref{Eq: Cost_known_a} with known $\mf{a}$ is 
$J (\mf{q}, \mf{a}) = \mf{a}^TX^*(\Sigma_V, T_0, T) \mf{a} + Y^*(\Sigma_V, T_0, T).$
Hence, the multiplicative regret for a Bayesian strategy with prior covariance $\Sigma$ is 
\bql\label{Eq: Multiplicative_Regret}
\operatorname{MR}(\mf{a}) = \dfrac{\mf{a}^TX(\Sigma, \Sigma_V, T_0, T) \mf{a} + Y(\Sigma, \Sigma_V, T_0, T)}{\mf{a}^TX^*(\Sigma_V, T_0, T)\mf{a} + Y^*(T, T_0, \Sigma_V)}.
\eql
We want to find $\Sigma$ such that \eqref{Eq: Multiplicative_Regret} becomes independent of $\mf{a}$. This occurs precisely when 
\bql \label{Eq: Matrix_Integral_Equation}
X(\Sigma, \Sigma_V, T_0, T) = \dfrac{Y(\Sigma, \Sigma_V, T_0, T)}{Y^*(\Sigma_V, T_0, T)}X^*(\Sigma_V, T_0, T).
\eql
Eq. \eqref{Eq: Matrix_Integral_Equation} is a matrix integral equation that can be solved for each $T$ numerically using a root-finding method. 
In fact, this equation has an intuitive interpretation. At optimality, the worst-case multiplicative regret is equal to both the multiplicative regret at $\mf{a} = 0$, given by $\frac{Y}{Y^*}$, and the multiplicative regret in the limit as $||\mf{a}||\to\infty$, which is given by the constant ratio of the $X$ and $X^*$ matrices. That is, the optimal agnostic strategy perfectly balances the regret for no drift and very large drift. Moreover, the equality of these multiplicative regrets is a sufficient condition for optimality.

For illustration, we identified the multiplicative regret-minimizing Bayesian strategy for a scalar system with dynamics \eqref{Eq: Continuous_Dynamics}, $\Sigma_W = 1$, $\Sigma_V = \sigma_V^2$, $Q=R=1$, and $T_0 = 0$. The optimal prior covariance $\Sigma = \sigma^2$ was computed by solving \eqref{Eq: Matrix_Integral_Equation} numerically using a Newton solver. The optimal prior standard deviation $\sigma$ and the worst-case multiplicative regret $\operatorname{MR}^*$ for different values of sensor noise standard deviation $\sigma_V$ are shown in Fig.~\ref{Fig: MultiplicativeRegret}. For small time horizons, the optimal value of $\sigma$ is quite large, reflecting the fact that there is very little time to learn the dynamics of the system and thus the prior needs to be flexible in order for the controller to adapt quickly. This effect is particularly pronounced when the sensor noise $\Sigma_V$ is large, which would otherwise cause the controller to learn the dynamics too slowly. With a large time horizon, the controller can act more conservatively as there is abundant time to stabilize the system and poor performance at small times is only weakly penalized. It can be shown that the optimal prior standard deviation $\sigma$ tends to zero as $T\rightarrow\infty$. As expected, the worst-case regret increases with the amount of sensor noise. It peaks at $T\approx 3.5$ and tends to unity as $T$ tends to $0$ or $\infty$.

\subsection{Minimizing Additive Regret}
When minimizing the additive regret, we must start at a nonzero time $T_0$ for
reasons to be explained shortly. The additive regret is given by
\bql
\operatorname{AR}(\mf{a}) = \mf{a}^T \bigg[X(\Sigma, \Sigma_V, T, T_0) - X^*(\Sigma_V, T, T_0)\bigg]\mf{a} + \bigg[Y(\Sigma, \Sigma_V, T, T_0) - Y^*(\Sigma_V, T, T_0)\bigg].
\eql
We want to find a prior covariance $\Sigma$ such that $\operatorname{AR}(\mf{a})$ is independent of $\mf{a}$. It turns out that this occurs for $\Sigma = rI$ in the limit $r \to +\infty$. 
\begin{theorem}\label{Thm: Regret_Inedep_a}
For $0\leq T_0 < T$, $\underset{\Sigma\rightarrow\infty}{\lim} \mf{a}^T\bigg[X(\Sigma, \Sigma_V, T, T_0) - X^*(\Sigma_V, T, T_0)\bigg]\mf{a} = 0$. Moreover, $\underset{\Sigma\rightarrow\infty}{\lim} Y(\Sigma, \Sigma_V, T, T_0) - Y^*(\Sigma_V, T, T_0) \to \infty$ no slower than logarithmically as $T_0\rightarrow 0$ for fixed~$T$.
\end{theorem}
\begin{proof}
The proof for the scalar case $d=1$ is given in Appendix \ref{App: Additive_Regret}. The vector case is more involved but conceptually analogous.
\end{proof}

\begin{figure}[t]
\centering
\subfigure{\includegraphics[trim=0cm 0cm 0cm 0cm, clip=true, width=0.5\textwidth]{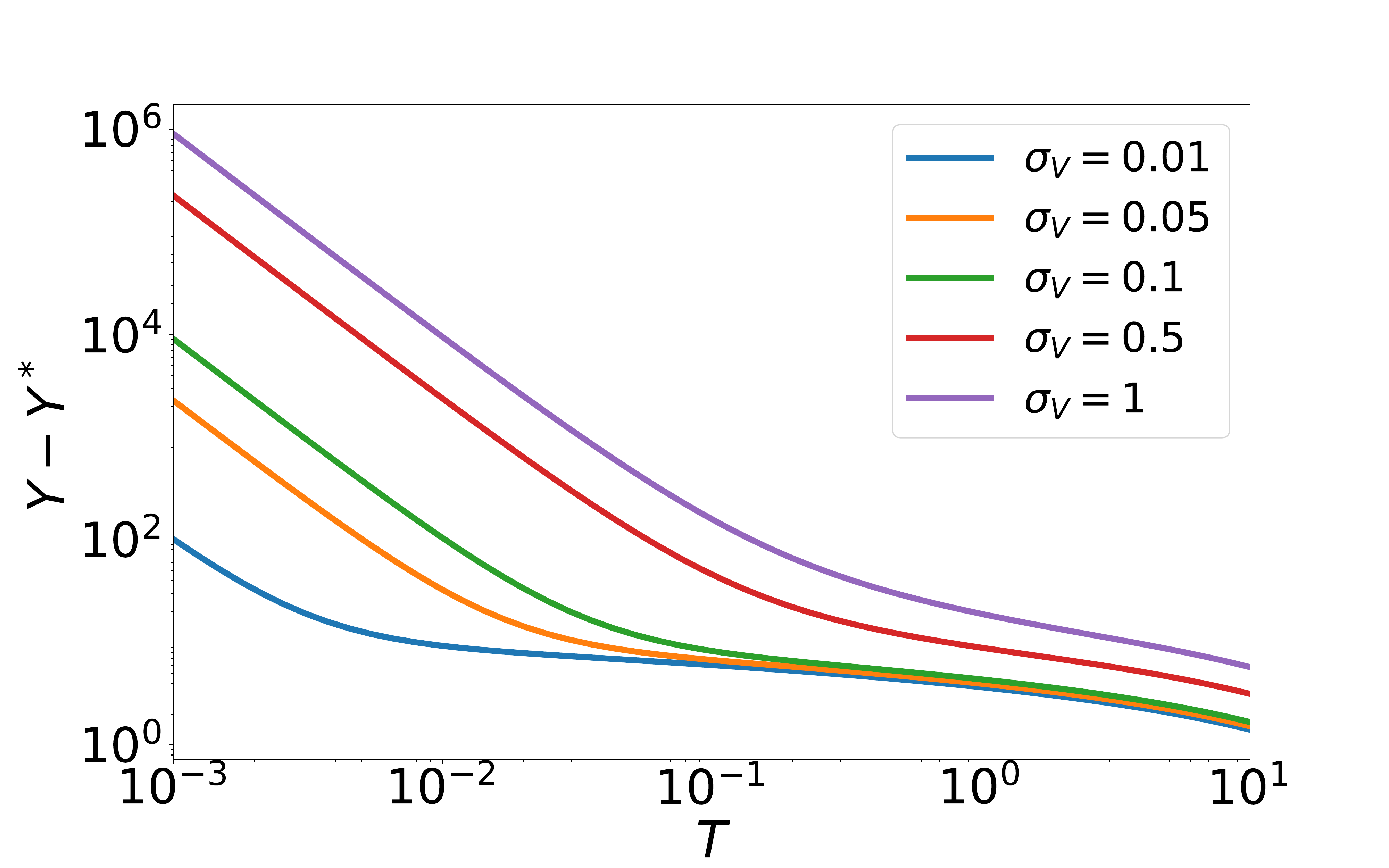}}
\subfigure{\includegraphics[trim=0cm 0cm 0cm 0cm, clip=true, width=0.5\textwidth]{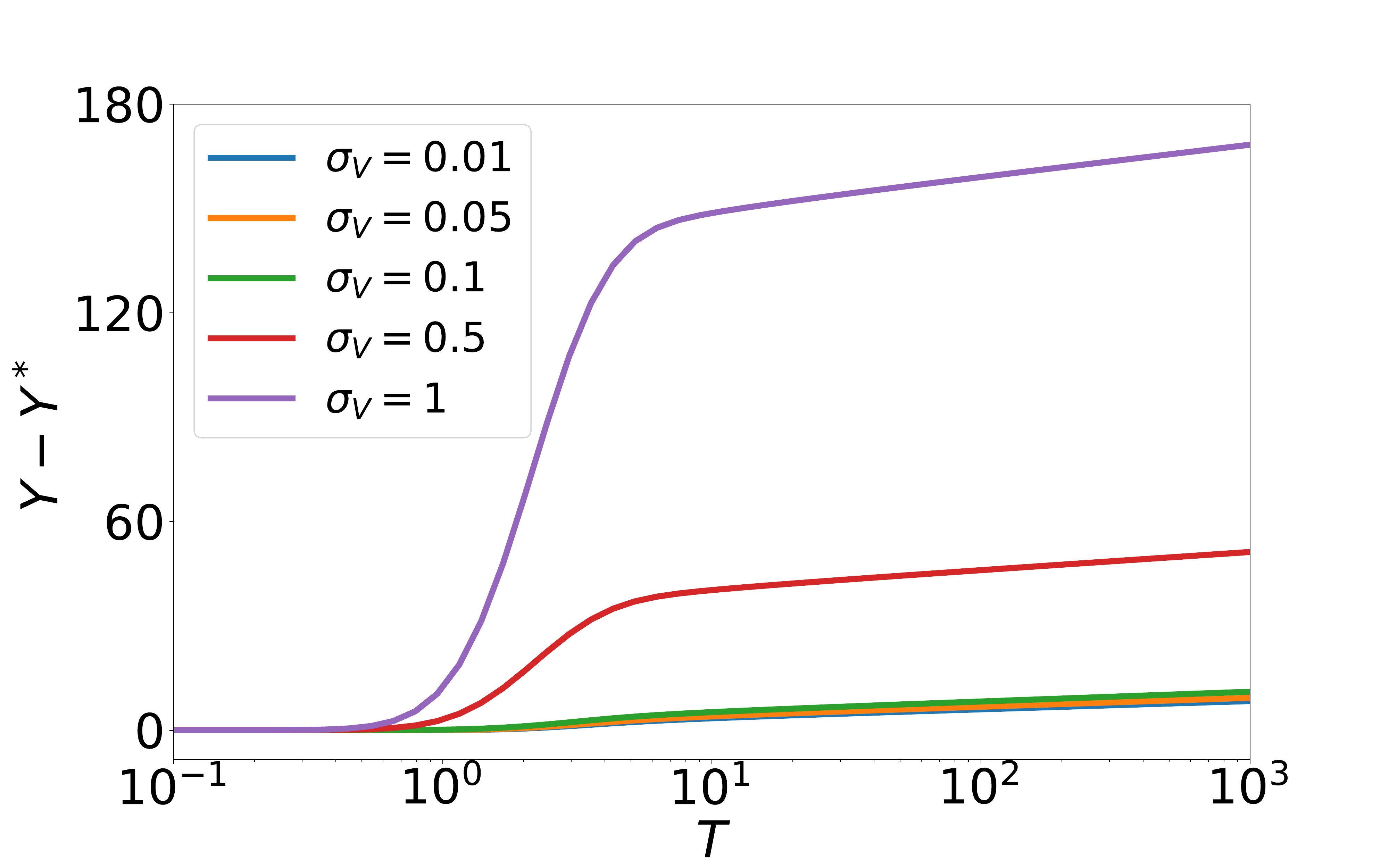}}
\caption{The optimal worst-case additive regret $Y-Y^*$ with
  varying standard deviation $\sigma_V$ of the sensor noise for (a) fixed $T =
  100$ and varying $T_0$ and (b) fixed $T_0 = 0.1$ and varying $T$. \label{Fig: AdditiveRegret}}
\end{figure}

Theorem \ref{Thm: Regret_Inedep_a} shows that the limit of Bayesian strategies with the prior covariance $\Sigma$ diverging to infinity optimizes the worst-case additive regret. Moreover, the additive regret in this limit is simply given by $\lim_{\Sigma\to\infty} Y(\Sigma, \Sigma_V, T, T_0) - Y^*(\Sigma_V, T, T_0)$. However, this quantity diverges as $T_0 \rightarrow 0$, so we must set $T_0>0$ to obtain a meaningful result. (Recall that we require that the control $\mf{u}(t)$ be set to zero for $t\in[0,T_0]$.) It remains to identify the specific strategy that arises as the limit of the Bayesian strategies. This requires knowledge of the optimal estimates $\mf{\hat{a}}$ and $\mf{\hat{q}}$ that appear in \eqref{Eq: Optimal_Control}. The challenge is that we cannot obtain these estimates directly from the ODEs \eqref{Eq: Estimation_ODE} as we did for the Bayesian strategies: these ODEs are singular in the limit as $\Sigma \to \infty$ and $t\to 0$. Instead, we can compute the optimal estimates from first principles. The result is
\begin{equation}
    \mf{\hat{a}}(t) = 2t\left(\int_0^t s^2 \omega_0(s)\,ds + t^2\kappa \right)^{-1}\left(\int_0^t \omega_0(s){\bf y}(s)\,ds+\kappa{\bf y}(t) \right), \quad \mf{\hat{q}}(t) = t\mf{\hat{a}}(t)
\end{equation}
where the matrices $\omega_0(s)$, $\kappa$ solve the system of integral equations
\begin{align}
    &\Sigma_V\omega_0(s)-\frac{\Sigma_W}{2T}\int_0^t \tau^2\omega_0(\tau)\,d\tau +\Sigma_W\int_s^t(\tau-t)\omega_0(\tau)\,d\tau = 
    I+\frac{(2s-t)}{2}\Sigma_W\kappa,\\\nonumber& \forall s \in [0, t) \text{ and }
    \int_0^t \omega_0(\tau)\,d\tau = -\kappa.
\end{align}
We can approximate the solution of the integral equations to arbitrary precision by discretizing them with respect to time and then solving the resulting matrix equation. Fig.~\ref{Fig: AdditiveRegret} shows plots of the additive regret for the regret-minimizing agnostic strategies corresponding to the scalar systems with parameters listed in Section \ref{Sec: Mult_Regret}. 
Panel (a) shows that the additive regret indeed diverges as $T_0 \to 0$ for fixed $T$, following a power law. Moreover, panel (b) indicates that the growth rate of the regret as $T \to \infty$ is logarithmic for fixed $T_0$.

\section{Conclusions}
\label{Sec: Conclusions}
We have identified optimal strategies for several variations of a control
problem based on a stochastic dynamical system with both process noise and sensor noise: the classical problem where the drift $\mf{a}$ is known, a Bayesian problem where the prior distribution of $\mf{a}$ is normal, and the agnostic problem where $\mf{a}$ is entirely unknown and we wish to minimize the worst-case multiplicative or additive regret. This last problem is of particular practical interest, as it provides valuable intuition about how a controller should simultaneously learn and control a system without prior knowledge of the parameters. In this case, we were able to identify the optimal agnostic strategy as a Bayesian strategy or limit of Bayesian strategies arising from a Gaussian prior. Qualitatively, our results show that a wide prior is necessary to minimize the worst-case regret on a short time horizon. In fact, the worst-case additive regret is minimized as the width of the prior diverges to infinity.

As a future research direction, it is worth exploring whether it is always possible to find a Bayesian strategy which is also an optimal or near-optimal agnostic strategy. For instance, \cite{carruth2021} treats the scalar system with the state dynamics $\Delta q = (aq+u)\Delta t+\Delta W,$. The aforementioned work provides further theoretical results supporting the hypothesis that optimal agnostic strategies for a wide variety of control problems naturally arise from Bayesian strategies.

\appendix
\section{Proof of Theorem \ref{Thm: Regret_Bounds}}\label{App: Regret_Bounds}
\begin{proof}
 The upper bounds are trivial, since the optimal worst-case regret cannot exceed the worst-case regret of any given strategy. It remains to prove the lower bounds. The expected costs of $\mc{A}$ and $\mc{B}$ under the prior $\mu$ are given by
\beq
  J_{\mc{A},\mu} = \int J_{\mc A}(\mf{a})\,d\mu(\mf{a}), \quad J_{\mc{B},\mu} = \int J_{\mc B}(\mf{a})\,d\mu(\mf{a}).
\eeq
These costs are related by the inequality $J_{\mc{A},\mu} \leq J_{\mc{B},\mu}$ by the Bayesian optimality of $\mc{A}$ under $\mu$. However, $J_{\mc{A}}(\mf{a}) = J_o(\mf{a})+\operatorname{AR}_{\mc{A}}(\mf{a})$ and similarly for $\mc{B}$, so
\beq
  \int \operatorname{AR}_{\mc A}(\mf{a})\,d\mu(\mf{a}) \leq \int \operatorname{AR}_{\mc B}(\mf{a})\,d\mu(\mf{a}).
\eeq
In turn, $\operatorname{AR}_{\mc B}(\mf{a})$ is bounded from above by $\operatorname{AR}^*_{\mc B}$, so since $\int d\mu(\mf{a}) = 1$, we obtain the lower bound in \eqref{Eq: AR_Bound}. Moreover, note that $J_{\mc{A}}(\mf{a}) = \operatorname{MR}_{\mc{A}}(\mf{a})J_o(\mf{a})$ and similarly for $\mc{B}$. Hence we have
\begin{equation}
    \int \operatorname{MR}_{\mc{A}}(\mf{a})J_o(\mf{a})\,d\mu(\mf{a}) \leq \int \operatorname{MR}_{\mc{B}}(\mf{a})J_o(\mf{a})\,d\mu(\mf{a}) \leq \operatorname{MR}^*_\mc{B} \int J_o(\mf{a})\,d\mu(\mf{a}),
\end{equation}
which yields the lower bound in \eqref{Eq: MR_Bound}.
\end{proof}

\section{Proof of Theorem \ref{Thm: Cost_of_a_strategy}}\label{App: Cost_of_a_strategy}
Here we prove Theorem \ref{Thm: Cost_of_a_strategy}, expressing the cost incurred by a Bayesian strategy as a quadratic form of $\mf{a}$.
\begin{proof}
By \eqref{Eq: Cost_strategy_breakdown}, the cost incurred by a Bayesian strategy is
\begin{align}
\mc{J}(\mf{q}, \mf{a}; \mu)
&= \int\limits_{T_0}^T\left[ \check{\mf{q}}^TQ\check{\mf{q}} + (S_{11}\bar{\mf{q}} + S_{12}\bar{\mf{a}})^TR^{-1}(S_{11}\bar{\mf{q}} + S_{12}\bar{\mf{a}})\right]dt\\\nonumber
& + \int\limits_{T_0}^T\operatorname{Trace}\bigg(Q\Sigma_{\mf{q}\hat{\mf{q}}} + S_{11}^TR^{-1}S_{11}\Sigma_{\hat{\mf{q}}\hat{\mf{q}}} 
 + 2S_{11}^TR^{-1}S_{12}\Sigma_{\hat{\mf{q}}\hat{\mf{a}}} + S_{12}^TR^{-1}S_{12}\Sigma_{\hat{\mf{a}}\hat{\mf{a}}}  \bigg)dt.
\end{align}
To compute this cost, we need to derive the ODEs for the aforementioned mean and covariances from the SDEs \eqref{Eq: SDE_Strategy}. This can be done by computing the expectations for the discretized dynamics with a time step $\Delta t$ and taking the limit as $\Delta t \rightarrow 0$. For example,
\bqn \Delta \Bar{\mf{a}} &=& \E[\Delta\mf{a}]\\ &=& \E\left[P_{12}\Sigma_V^{-1}\left[(\mf{q} - \hat{\mf{q}})\Delta t +\Delta \mf{V}\right]\right]\\ &=& P_{12}\Sigma_V^{-1}\left[(\check{\mf{q}} - \Bar{\mf{q}})\Delta t\right].\eqn
Therefore, \bql\dot{\Bar{\mf{a}}} = P_{12}\Sigma_V^{-1}(\check{\mf{q}} - \Bar{\mf{q}}).\eql
Again,
\bqn 
\Sigma_{\hat{\mf{a}}\hat{\mf{a}}}(t + \Delta t)
&=& \E\left[\{\hat{\mf{a}}(t + \Delta t) - \Bar{\mf{a}}(t + \Delta t)\}\{\hat{\mf{a}}(t + \Delta t) - \Bar{\mf{a}}(t + \Delta t)\}^T\right]\\
&=& \E\bigg[ (\hat{\mf{a}} - \Bar{\mf{a}} + P_{12}\Sigma_V^{-1}\{(\mf{q} - \check{\mf{q}}) - (\hat{\mf{q}} - \Bar{\mf{q}})\}\Delta t + P_{12}\Sigma_V^{-1}\Delta\mf{V})\\
&& (\hat{\mf{a}} - \Bar{\mf{a}} + P_{12}\Sigma_V^{-1}\{(\mf{q} - \check{\mf{q}}) - (\hat{\mf{q}} - \Bar{\mf{q}})\}\Delta t + P_{12}\Sigma_V^{-1}\Delta\mf{V})^T\bigg]\\
&=& \Sigma_{\hat{\mf{a}}\hat{\mf{a}}}(t) + \bigg[(\Sigma_{\mf{q}\hat{\mf{a}}} - \Sigma_{\hat{\mf{q}}\hat{\mf{a}}})^T\Sigma_V^{-1}P_{12} + P_{12}\Sigma_V^{-1}(\Sigma_{\mf{q}\hat{\mf{a}}} - \Sigma_{\hat{\mf{q}}\hat{\mf{a}}}) + P_{12}\Sigma_V^{-1}P_{12}\bigg]\Delta t,
\eqn
i.e.,
\bql
\dot{\Sigma}_{\hat{\mf{a}}\hat{\mf{a}}} = (\Sigma_{\mf{q}\hat{\mf{a}}} - \Sigma_{\hat{\mf{q}}\hat{\mf{a}}})^T\Sigma_V^{-1}P_{12} + P_{12}\Sigma_V^{-1}(\Sigma_{\mf{q}\hat{\mf{a}}} - \Sigma_{\hat{\mf{q}}\hat{\mf{a}}}) + P_{12}\Sigma_V^{-1}P_{12}.
\eql 
Similarly we obtain the ODEs
\begin{align}
\dot{\check{\mf{q}}} &= \mf{a} - R^{-1}(S_{11}\Bar{\mf{q}} - S_{12}\Bar{\mf{a}})\\\nonumber
\dot{\Bar{\mf{q}}} &= \Bar{\mf{a}} - R^{-1}(S_{11}\Bar{\mf{q}} - S_{12}\Bar{\mf{a}}) + P_{11}\Sigma_V^{-1}(\check{\mf{q}} - \Bar{\mf{q}})\\\nonumber 
\dot{\Bar{\mf{a}}} &= P_{12}\Sigma_V^{-1}(\check{\mf{q}} - \Bar{\mf{q}})\\\nonumber
\dot{\Sigma}_{\mf{q}\mf{q}} &= \Sigma_W - (\Sigma_{\mf{q}\hat{\mf{q}}}S_{11} + \Sigma_{\mf{q}\hat{\mf{a}}}S_{12}^T)R^{-1} - R^{-1}(\Sigma_{\mf{q}\hat{\mf{q}}}S_{11} + \Sigma_{\mf{q}\hat{\mf{a}}}S_{12}^T)^T\\\nonumber
\dot{\Sigma}_{\mf{q}\hat{\mf{q}}}&= (\Sigma_{\mf{q}\mf{q}} - \Sigma_{\mf{q}\hat{\mf{q}}})\Sigma_V^{-1}P_{11} - \Sigma_{\mf{q}\hat{\mf{q}}}S_{11}^TR^{-1} +\Sigma_{\mf{q}\hat{\mf{a}}}(I - R^{-1}S_{12})^T - R^{-1}(S_{11}\Sigma_{\hat{\mf{q}}\hat{\mf{q}}} + S_{12}\Sigma_{\hat{\mf{q}}\hat{\mf{a}}}^T)\\\nonumber
\dot{\Sigma}_{\mf{q}\hat{\mf{a}}} &= (\Sigma_{\mf{q}\mf{q}} - \Sigma_{\mf{q}\hat{\mf{q}}})\Sigma_V^{-1}P_{12} - R^{-1}(S_{11}\Sigma_{\hat{\mf{q}}\hat{\mf{a}}} + S_{12}\Sigma_{\hat{\mf{a}}\hat{\mf{a}}}^T)\\\nonumber
\dot{\Sigma}_{\hat{\mf{q}}\hat{\mf{q}}}&= P_{11}^T\Sigma_V^{-1}P_{11} + \Sigma_{\hat{\mf{q}}\hat{\mf{a}}}(I -R^{-1}S_{12})^T + (I -R^{-1}S_{12})\Sigma_{\hat{\mf{q}}\hat{\mf{a}}}^T + (\Sigma_{\mf{q}\hat{\mf{q}}} - \Sigma_{\hat{\mf{q}}\hat{\mf{q}}})^T\Sigma_V^{-1}P_{11}\\\nonumber & + P_{11}\Sigma_V^{-1}(\Sigma_{\mf{q}\hat{\mf{q}}} - \Sigma_{\hat{\mf{q}}\hat{\mf{q}}}) - \Sigma_{\hat{\mf{q}}\hat{\mf{q}}}S_{11}R^{-1} - R^{-1}S_{11}\Sigma_{\hat{\mf{q}}\hat{\mf{q}}}\\\nonumber
\dot{\Sigma}_{\hat{\mf{q}}\hat{\mf{a}}}&= (\Sigma_{\mf{q}\hat{\mf{q}}} - \Sigma_{\hat{\mf{q}}\hat{\mf{q}}})^T\Sigma_V^{-1}P_{12} + (I - R^{-1}S_{12})\Sigma_{\hat{\mf{a}}\hat{\mf{a}}}  - R^{-1}S_{12}\Sigma_{\hat{\mf{q}}\hat{\mf{a}}} + P_{11}\Sigma_V^{-1}(\Sigma_{\mf{q}\hat{\mf{a}}} - \Sigma_{\hat{\mf{q}}\hat{\mf{a}}} + P_{12})\\\nonumber
\dot{\Sigma}_{\hat{\mf{a}}\hat{\mf{a}}}&= (\Sigma_{\mf{q}\hat{\mf{a}}} - \Sigma_{\hat{\mf{q}}\hat{\mf{a}}})^T\Sigma_V^{-1}P_{12} + P_{12}\Sigma_V^{-1}(\Sigma_{\mf{q}\hat{\mf{a}}} - \Sigma_{\hat{\mf{q}}\hat{\mf{a}}}) + P_{12}\Sigma_V^{-1}P_{12}.
\end{align}
The ODEs for $\check{\mf{q}}$, $\bar{\mf{q}}$, and $\bar{\mf{a}}$ yield $\check{\mf{q}} = C_1(t)\mf{a}$, $\bar{\mf{q}} = C_2(t)\mf{a}$ and $\bar{\mf{a}}=C_3(t)\mf{a}$ for some $C_1,\,C_2,\,C_3:[0,\, T] \rightarrow \R^{d\times d}$ that satisfy the same set of ODEs. Hence the incurred cost \eqref{Eq: Cost_strategy_breakdown} for a particular Bayesian strategy becomes
\begin{align} \label{Eq: cost_strategy_ultimate}
\mc{J}(\mf{q}, \mf{a}; \mu)
&= \E\left[\;\int\limits_{T_0}^T \left( \mf{q}^T Q\mf{q} + \mf{u}^T R\mf{u}\right)dt\right]\\\nonumber
&= \int\limits_{T_0}^T\mf{a}^T\bigg[ C_1^TQC_1 +  (S_{11}C_2 + S_{12}C_3)^TR^{-1}(S_{11}C_2 + S_{12}C_3)\bigg]\mf{a}\,dt\\\nonumber
& + \int\limits_{T_0}^T\operatorname{Trace}\bigg(Q\Sigma_{\mf{q}{\mf{q}}} + S_{11}^TR^{-1}S_{11}\Sigma_{\hat{\mf{q}}\hat{\mf{q}}} + 2S_{11}^TR^{-1}S_{12}\Sigma_{\hat{\mf{q}}\hat{\mf{a}}} + S_{12}^TR^{-1}S_{12}\Sigma_{\hat{\mf{a}}\hat{\mf{a}}}  \bigg)dt \\\nonumber
&= \mf{a}^TX(\Sigma, \Sigma_V, T_0, T) \mf{a} + Y(\Sigma, \Sigma_V, T_0, T),
\end{align}
where 
\beq X(\Sigma, \Sigma_V, T_0, T) = \int\limits_{T_0}^T\bigg[ C_1^TQC_1 + (S_{11}C_2 + S_{12}C_3)^TR^{-1}(S_{11}C_2 + S_{12}C_3)\bigg]dt\eeq
and \beq Y(\Sigma, \Sigma_V, T_0, T) = \int\limits_{T_0}^T\operatorname{Trace}\bigg(Q\Sigma_{\mf{q}\mf{q}} + S_{11}^TR^{-1}S_{11}\Sigma_{\hat{\mf{q}}\hat{\mf{q}}} + 2S_{11}^TR^{-1}S_{12}\Sigma_{\hat{\mf{q}}\hat{\mf{a}}} + S_{12}^TR^{-1}S_{12}\Sigma_{\hat{\mf{a}}\hat{\mf{a}}}  \bigg)dt. \eeq
This completes the proof of Theorem \ref{Thm: Cost_of_a_strategy}. 
\end{proof}

\section{Proof of Theorem \ref{Thm: Cost_for_known_a}} \label{App: Cost_for_known_a}
Theorem \ref{Thm: Cost_for_known_a} is proved in a similar fashion to Theorem \ref{Thm: Cost_of_a_strategy}, i.e., by deriving the ODEs for $\Check{\mf{q}}^*$, $\Bar{\mf{q}}^*$, and the corresponding covariances.
\begin{proof}
From the Riccati equation \eqref{Eq: Estimation_Riccati_known_a} and the optimal control \[\mf{u}^*(t) = -R^{-1}(S_{11}\hat{\mf{q}}^*(t) + S_{12}{\mf{a}}(t)),\]the ODEs \eqref{Eq: ODE_Strategy} becomes
\begin{align}
\dot{\check{\mf{q}}}^* &= \mf{a} - R^{-1}(S_{11}\Bar{\mf{q}}^* - S_{12}{\mf{a}})\\\nonumber
\dot{\Bar{\mf{q}}}^* &= {\mf{a}} - R^{-1}(S_{11}\Bar{\mf{q}}^* - S_{12}{\mf{a}}) + P_{11}\Sigma_V^{-1}(\check{\mf{q}}^* - \Bar{\mf{q}}^*)\\\nonumber 
\dot{\Sigma}_{\mf{q}\mf{q}}^* &= \Sigma_W - (\Sigma_{\mf{q}\hat{\mf{q}}}^*S_{11}R^{-1} + R^{-1}S_{11}\Sigma_{\mf{q}\hat{\mf{q}}}^{*^T})\\\nonumber
\dot{\Sigma}_{\mf{q}\hat{\mf{q}}}^*&= (\Sigma_{\mf{q}\mf{q}}^* - \Sigma_{\mf{q}\hat{\mf{q}}}^*)\Sigma_V^{-1}P_{11} - \Sigma_{\mf{q}\hat{\mf{q}}}^*S_{11}^TR^{-1} - R^{-1}S_{11}\Sigma_{\hat{\mf{q}}\hat{\mf{q}}}^* \\\nonumber
\dot{\Sigma}_{\hat{\mf{q}}\hat{\mf{q}}}^*&= P_{11}^{*^T}\Sigma_V^{-1}P_{11} + (\Sigma_{\mf{q}\hat{\mf{q}}}^* - \Sigma_{\hat{\mf{q}}\hat{\mf{q}}}^*)^T\Sigma_V^{-1}P_{11}^* + P_{11}^*\Sigma_V^{-1}(\Sigma_{\mf{q}\hat{\mf{q}}}^* - \Sigma_{\hat{\mf{q}}\hat{\mf{q}}}^*) \\\nonumber &- \Sigma_{\hat{\mf{q}}\hat{\mf{q}}}^*S_{11}R^{-1} - R^{-1}S_{11}\Sigma_{\hat{\mf{q}}\hat{\mf{q}}}^*
\end{align}
The ODEs for $\check{\mf{q}}^*$ and $\bar{\mf{q}}$ yield $\check{\mf{q}} = C^*_1(t)\mf{a}$ and $\bar{\mf{q}} = C^*_2(t)\mf{a}$ for some $C^*_1,\,C^*_2:[0,\, T] \rightarrow \R^{d\times d}$ that satisfy the same set of ODEs. Proceeding in the same fashion as in \eqref{Eq: Cost_strategy_breakdown} and \eqref{Eq: cost_strategy_ultimate}, the optimal cost for the Bayesian strategy with known $\mf{a}$ becomes
\begin{align} 
J(\mf{q}, \mf{a})
&=\E\left[\;\int\limits_{T_0}^T \left( \mf{q}^{*^T} Q\mf{q}^* + \mf{u}^{*^T} R\mf{u}^*\right)dt\right]\\\nonumber
&= \int\limits_{T_0}^T\mf{a}^T\bigg[ C_1^{*^T}QC^*_1 +  (S_{11}C^*_2 + S_{12})^TR^{-1}(S_{11}C^*_2 + S_{12})\bigg]\mf{a}\,dt\\\nonumber
& + \int\limits_{T_0}^T\operatorname{Trace}\bigg(Q\Sigma_{\mf{q}\hat{\mf{q}}}^* + S_{11}^TR^{-1}S_{11}\Sigma_{\hat{\mf{q}}\hat{\mf{q}}}^* \bigg)dt \\\nonumber
&= \mf{a}^TX^*(\Sigma_V, T_0, T) \mf{a} + Y^*(\Sigma_V, T_0, T),
\end{align} 
where 
\beq X^*(\Sigma_V, T_0, T) = \int\limits_{T_0}^T\bigg[C_1^{*^T}QC^*_1 + (S_{11}C^*_2 + S_{12})^TR^{-1}(S_{11}C^*_2 + S_{12})\bigg]dt\eeq
and \beq Y^*(\Sigma_V, T_0, T) = \int\limits_{T_0}^T\operatorname{Trace}\bigg(Q\Sigma_{\mf{q}\mf{q}}^* + S_{11}^TR^{-1}S_{11}\Sigma_{\hat{\mf{q}}\hat{\mf{q}}}^*\bigg)dt. \eeq 
This completes the proof of Theorem \ref{Thm: Cost_for_known_a}.
\end{proof}
\section{Proof of Theorem \ref{Thm: Regret_Inedep_a} for the scalar case} \label{App: Additive_Regret}
Here we prove the Theorem \ref{Thm: Regret_Inedep_a} for $q,\,a\in\R$. We drop the boldface notation when considering the scalar variables. We use the notation $\Sigma = \sigma^2>0$, $\Sigma_W = \sigma_w^2>0$, $\Sigma_V = \sigma_v^2>0$, and we assume for simplicity that $Q=R=1$.
\begin{proof}
From \eqref{Eq: cost_strategy_ultimate}, \eqref{Eq: ODE_Strategy}, and \eqref{Eq: Estimation_Riccati}
\bql a^2X(\sigma, \sigma_v^2, T_0, T) = \int\limits_{T_0}^T \left[\check{q}^2 + (S_{12}\bar{a} + S_{11}\bar{q})^2\right],\eql
where
\bnl\label{Eq: scalar_strategy}
\dot{\check{q}} &=& a - (S_{12}\bar{a} + S_{11}\bar{q}),\, \check{q}(0)=0\\\nonumber
\dot{\bar{a}} &=& \dfrac{P_{12}}{\sigma_v^2} (\check{q}-\bar{q}),\, \bar{a}(0)=0\\\nonumber
\dot{\bar{q}} &=& (1 - S_{12})\bar{a} - S_{11}\bar{q} + \dfrac{P_{11}}{\sigma_v^2} (\check{q}-\bar{q}),\,\bar{q}(0)=0\\\nonumber
\dot{P}_{11} &=& 2P_{12} - \dfrac{P_{11}^2}{\sigma_v^2} + \sigma_w^2,\,P_{11}(0) = \sigma_{q_0}^2\\\nonumber
\dot{P}_{12} &=& P_{22} - \dfrac{P_{11}P_{12}}{\sigma_v^2},\,P_{12}=0\\\nonumber
\dot{P}_{22} &=& -\dfrac{P_{12}^2}{\sigma_v^2},\, P_{22}(0) = \sigma^2.
\enl
Similarly, from \eqref{Eq: Cost_known_a}, \eqref{Eq: ODE_Strategy_Known_a}, and \eqref{Eq: Estimation_Riccati_known_a}
\bql a^2X^*(\sigma_v^2, T_0, T) = \int\limits_{T_0}^T \left[\check{q}^{*^2} + (S_{12}a + S_{11}\bar{q}^*)^2\right]dt,\eql
where
\begin{align}\label{Eq: scalar_strategy_known_a}
\dot{\check{q}}^* &= a - (S_{12}\bar{a}^* + S_{11}\bar{q}^*),\, \check{q}^*(0)=0\\\nonumber
\dot{\bar{q}}^* &= (1 - S_{12})a - S_{11}\bar{q}^* + \dfrac{P^*_{11}}{\sigma_v^2} (\check{q}^*-\bar{q}^*),\,\bar{q}^*(0)=0\\\nonumber
\dot{P}^*_{11} &= - \dfrac{P_{11}^{*^2}}{\sigma_v^2} + \sigma_w^2,\,P^*_{11}(0) = \sigma_{q_0}^2.
\end{align}
Now, we wish to show that as $\sigma \rightarrow \infty$, the solutions of two set of ODEs \eqref{Eq: scalar_strategy} and \eqref{Eq: scalar_strategy_known_a} become equivalent. Notice that, if $\bar{q}(t)\equiv\check{q}(t)$ and $\bar{a}(t)\equiv a$ for all $t>0$ as $\sigma\rightarrow \infty$, then $\bar{q}(t)\equiv\check{q}(t)\equiv\check{q}^*(t)\equiv\bar{q}^*(t)$ for $t>0$ if $\check{q}^*(0)=\bar{q}^*(0)=0$. This is confirmed by the following lemma.
\begin{lemma}
As $\sigma\rightarrow\infty$, $\bar{q}(t)\rightarrow\check{q}(t)$ and $\bar{a}(t)\rightarrow a$ for $t>0$.
\end{lemma}
\begin{proof}
Let $e_a \triangleq \bar{a} - a$ and $e_q \triangleq \bar{q} - \check{q}$. The first three ODEs in \eqref{Eq: scalar_strategy} can be expressed as
\bql\label{Eq: e} \dot{e} = \begin{bmatrix}\sigma_v^{-2}P_{11} & 1 \\\sigma_v^{-2}P_{12} & 0  \end{bmatrix}e,\quad e(0) = \begin{bmatrix}0 \\ -a\end{bmatrix}\eql where $e = [e_q\quad e_a]^T$. To illustrate the proof, we will take $\sigma_w=\sigma_v=\sigma_{q_0} = 1$. After some algebraic manipulation,
\begin{align}\label{Eq: P_solution_scalar}
P(t) &= \begin{bmatrix} P_{11}(t) & P_{12}(t) \\ P_{12}(t) & P_{22}(t) \end{bmatrix} \\\nonumber &= \dfrac{1}{\sigma^2(2t - 3 + 4e^{-t} - e^{-2t}) + 2}\begin{bmatrix} (2t - 1 + e^{-2t})\sigma^2 + 2 & 2\sigma^2(1-e^{-t})\\ 2\sigma^2(1-e^{-t}) & 2\sigma^2\end{bmatrix}.
\end{align}
This allows us to solve the ODE \eqref{Eq: e}:
\bql e(t) = \dfrac{2a}{\sigma^2(2t - 3 + 4e^{-t} - e^{-2t}) + 2}\begin{bmatrix}1 - e^{-t} \\ 1\end{bmatrix}\eql
It is evident that $e(t)$ is continuous on $t\geq 0$. For fixed $t$, $e(t)$ decays at a rate proportional to $\sigma^{-2}$ as $\sigma\rightarrow\infty$. On the other hand, for fixed $\sigma$, the decay rate as $t\to\infty$ is proportional to $t^{-1}$. Hence, $e(t)$ converges pointwise to $0$ for all $t>0$ as $\sigma\rightarrow\infty$, i.e., $\bar{q}(t)\rightarrow\check{q}(t)$ and $\bar{a}(t)\rightarrow a$ pointwise for $t>0$ as $\sigma\rightarrow\infty$.
\end{proof}
Although we have assumed $\sigma_w=\sigma_v=\sigma_{q_0}=1$, the proof is readily generalized for arbitrary nonnegative values of these quantities at the cost of more involved algebraic computations. This completes the proof of the first statement of Theorem \ref{Thm: Regret_Inedep_a}.

For the second part of the theorem, notice that in the limit $\sigma \to \infty$, $P_{12}(t)$ scales as $t^{-1}$ for small positive $t$. Consider the right-hand side of the ODE for $\Sigma_{\hat{a}\hat{a}}$ in \eqref{Eq: ODE_Strategy}; this expression is independent of $\Sigma_{\hat{a}\hat{a}}$ and goes as $t^{-2}$ due to the term that is quadratic in $\Sigma_{\hat{a}\hat{a}}$. This makes $\Sigma_{\hat{a}\hat{a}}\sim \dfrac{C}{t}$ as $t\to 0^+$ for some constant $C$. From \eqref{Eq: cost_strategy_ultimate}, $Y(\sigma,\sigma_v,T_0, T)$ involves an integral of $\Sigma_{\hat{a}\hat{a}}$, which diverges logarithmically as $T_0\rightarrow 0$. Furthermore, all of the integrands are non-negative. 
This completes the proof of Theorem \ref{Thm: Regret_Inedep_a} for the scalar case. 
\end{proof}

The general proof for $\mf{q},\,\mf{a}\in\R^d$ requires solving for $P(t)$ and $\mf{e}(t)$ in $\R^{2d\times 2d}$ and $\R^{2d}$ respectively, but the asymptotics of the solutions remain the same. Hence, the results obtained in the scalar case will continue to hold. 

\section{Optimal estimates of ${\bf a}$ and ${\bf q}$} \label{App: Estimation}

In this appendix we derive the optimal estimates of ${\bf a}$ and ${\bf q}$ at a fixed time $T>0$. We observe ${\bf y}(t)$ for $t \in [0, T]$ and aim to compute the joint posterior probability distribution of $({\bf a}, {\bf q}(T))$ in the limit as the covariance matrix of the Gaussian prior on ${\bf a}$, $\Sigma$, diverges to infinity. We assume that the control input ${\bf u}(t)$ is zero on $[0, T]$.

Since we know that the distribution of $({\bf a}, {\bf q}(T))$ is jointly normal, the means of the marginal distributions of ${\bf a}$ and ${\bf q}(T)$ are the optimal estimates $\hat{\bf a}$ and $\hat{\bf q}$. By \eqref{Eq: Continuous_Dynamics}, we know $${\bf q}(t) = {\bf a}t+{\bf W}(t)$$ and $${\bf y}(t) = \int_0^t {\bf q}(t)\,dt + \int_0^t d{\bf V} = \frac{t^2}{2}{\bf a}+\int_0^t {\bf W}(t)\,dt + {\bf V}(t).$$
Let ${\bf \tilde{W}}(t) = {\bf W}(t)-\frac{t}{T}{\bf W}(T)$. Then ${\bf W}(T)$ and $({\bf W}(t)), t \in [0, T]$ are independent. Then $${\bf q}(t) = \left({\bf a}+\frac{{\bf W}(T)}{T}\right)t+\tilde{{\bf W}}(t)$$ and $${\bf y}(t) = \frac{1}{2}\left({\bf a}+\frac{{\bf W}(T)}{T}\right)t^2+\int_0^T\tilde{{\bf W}}(t)\,dt+{\bf V}(t).$$ Moreover we may compute
\begin{equation}
    \E\left[{\bf V}(t_1){\bf V}(t_2)\right] = \Sigma_V\min(t_1,t_2),
\end{equation}
\begin{equation}
    \E\left[{\bf W}(t_1){\bf W}(t_2)\right] = \Sigma_W\min(t_1,t_2),
\end{equation}
and
\begin{equation}
    \E\left[\tilde{\bf W}(t_1)\tilde{\bf W}(t_2)\right] = \Sigma_W\left(\min(t_1,t_2)-\frac{t_1t_2}{T}\right).
\end{equation}
We will let ${\bf X}(t) = \int_0^T \tilde{\bf W}(t)\,dt + {\bf V}(t)$, which is a Gaussian process on $[0, T]$ that is independent of both ${\bf a}$ and ${\bf W}(T)$.

In order to determine the joint probability distribution of $({\bf a}, {\bf q}(T))$ given the measurements of ${\bf y}$, we can start by finding the joint distribution of $({\bf a}, {\bf W}(T))$. To begin, we let $t_\nu = \nu h$ for $\nu = 1, \dots, N$ with $h = \frac{T}{N}$. Then, the joint probability density of $({\bf a}, {\bf W}(T), ({\bf X}(t))_{t = t_1, \dots, t_N})$ (not yet conditioned on the values of $y$) is given by
\begin{equation}
    dP = \rho({\bf a})\,d{\bf a} \frac{\exp\left[-\frac{1}{2T}{\bf W}(T)\Sigma_W^{-1}{\bf W}(T)\right]}{\sqrt{\det(2\pi T\Sigma_W)}}\,d{\bf W}(T)\frac{\exp\left[-\frac{1}{2}{\mathbbm X}(T)S^{-1}{\mathbbm X}(T)\right]}{\sqrt{\det(2\pi S)}}\,d{\bf X}_1\dots d{\bf X}_N
\end{equation}
where $\mathbbm{X}$ is the vector formed by concatenating ${\bf X}(t_1), \dots, {\bf X}(t_N)$ and $S_{mn} = \E[\mathbbm{X}_m \mathbbm{X}_n]$. It will be most convenient to consider $S$ as a matrix with blocks $S(t_\mu, t_\nu) = \E\left[{\bf X}(t_\mu){\bf X}(t_\nu)\right]$. Moreover, we will separate out the normalizing factors such as $\sqrt{\det(2\pi S)}$ which are independent of the random variables of interest. 

We know that ${\bf X}(t_\nu) = {\bf y}(t_\nu)-\frac{1}{2}\left({\bf a}+\frac{{\bf W}(T)}{T}\right)t_\nu^2 = {\bf y}(t_\nu)-\frac{t_\nu^2}{2T}{\bf q}(T)$. Let $\mathbbm{y}$ and $\mathbbm{R}$ be formed by vertically concatenating ${\bf y}$ and the matrices $\frac{t_\nu^2}{2T}I$ for $t = t_1, \dots, t_N$; then we may write $$\mathbbm{X} = \mathbbm{y}-\mathbbm{R}{\bf q}(T).$$ Thus
\begin{equation}
\begin{aligned}
    dP &= F(\mathbbm{y})\rho({\bf a})\,d{\bf a} \exp\left[-\frac{1}{2}{\bf W}(T)\Sigma_W^{-1}{\bf W}(T)\right]\,d{\bf W}(T) \\ &\times \exp\left[-\frac{1}{2}(\mathbbm{y}-\mathbbm{R}{\bf q}(T))^TS^{-1}(\mathbbm{y}-\mathbbm{R}{\bf q}(T))\right]\,d{\bf y}_1\dots d{\bf y}_N.
\end{aligned}
\end{equation}
Conditioning on $y$ and noting that $\rho({\bf a})$ becomes independent of ${\bf a}$ on any bounded set as $\Sigma \to \infty$:
\begin{align*}
      dP' &= F(\mathbbm{y})\exp\left[-\frac{1}{2}{\bf W}(T)^T\Sigma_W^{-1}{\bf W}(T)\right]\\
      & \times\exp\left[-\frac{1}{2}(\mathbbm{y}-\mathbbm{R}{\bf q}(T))^TS^{-1}(\mathbbm{y}-\mathbbm{R}{\bf q}(T))\right]d{\bf a}\,d{\bf W}(T)
\end{align*}
where $F(\mathbbm{y})$ is a normalizing factor.
We may change variables by noting $d{\bf a}\,d{\bf W}(T) = T^d d{\bf a}\,d{\bf q}(T)$ and ${\bf W}(T) = {\bf q}(T)-{\bf a}T$. Writing ${\bf q}$ as a shorthand for ${\bf q}(T)$ and expanding the parentheses in the second exponential, we obtain
\begin{equation}
    dP' = \tilde{F}(\mathbbm{y})\exp\left[-\frac{1}{2}({\bf q}-{\bf a}T)^T\Sigma_W^{-1}({\bf q}-{\bf a}T)-\frac{1}{2}{\bf q}^T\mathbbm{R}^TS^{-1}\mathbbm{R}{\bf q}+\mathbbm{y}^TS^{-1}\mathbbm{R}{\bf q}\right]d{\bf a}\,d{\bf q}
\end{equation}
with some normalization factor $\tilde{F}(\mathbbm{y})$.

We wish to compute $S^{-1}\mathbbm{R}$. We define the matrices $\omega(\nu)$ as the solution of the system of equations
\begin{equation} \label{Eq: Def_Omega}
    \sum_{\mu=1}^\nu \sigma(t_\nu, t_\mu)\omega(t_\mu) h = \frac{t_\nu^2}{2}I, \quad \nu = 1, \dots, N
\end{equation}
so that $S^{-1}\mathbbm{R}$ is formed by vertically concatenating the matrices $\omega(t_\nu)h$. Then
\begin{equation} \label{Eq: Term2}
    \mathbbm{y}^TS^{-1}\mathbbm{R}{\bf q} = \left(\sum_{\nu=1}^N \omega(t_\nu){\bf y}(t_\nu) h\right)^T \frac{{\bf q}}{T}
\end{equation}
and
\begin{equation} \label{Eq: Term3}
    \frac{1}{2}{\bf q}^T\mathbbm{R}^TS^{-1}\mathbbm{R}{\bf q} = \frac{1}{2T^2}{\bf q}^T \left (\sum_{\nu=1}^N \frac{t_\nu^2}{2}\omega(t_\nu)h \right){\bf q}.
\end{equation}
One may show that
\begin{equation}
    S(t_\nu, t_\mu) = \E[{\bf X}(t_\nu) {\bf X}(t_\mu)] = \begin{cases}\Sigma_V t_\nu+\Sigma_W\left(\frac{1}{2}t_\nu^2t_\mu-\frac{1}{6}t_\nu^3-\frac{1}{4T}t_\nu^2t_\mu^2\right) & t_\nu\leq t_\mu \\
    \Sigma_Vt_\mu+\Sigma_W\left(\frac{1}{2}t_\mu^2t_\nu-\frac{1}{6}t_\mu^3-\frac{1}{4T}t_\nu^2t_\mu^2\right) & t_\mu\leq t_\nu \end{cases}.
\end{equation}
Then, by taking the second forward difference of \eqref{Eq: Def_Omega} for $\nu = 1, \dots, N-1$, it follows that the matrices $\omega(t_\nu)$ must satisfy
\begin{equation}
\begin{aligned}
    \left(-\Sigma_V-\Sigma_W \frac{h^2}{6}\right)\omega(t_\nu) - \frac{\Sigma_W}{2T}\sum_{\mu=1}^{N-1} t_\mu^2 \omega(t_\mu)h + \Sigma_W\sum_{\mu=\nu+1}^{N-1} (t_\mu-t_\nu)\omega(t_\mu)h \\ = \frac{2t-T}{2}\Sigma_W\omega(T)h+I, \quad \nu = 1, \dots, N-1.
\end{aligned}
\end{equation}
We also have the constraint given by \eqref{Eq: Def_Omega} for $\nu = 1$, namely
\begin{equation}
    \sum_{\mu=1}^N \left[\Sigma_V+\left(\frac{t_\mu}{2}-\frac{t_\mu^2}{4}\right)\Sigma_W h-\frac{1}{6}\Sigma_W h^2\right]\omega(t_\mu)h = \frac{h}{2}.
\end{equation}

We conclude by considering the continuous limit $h = \frac{1}{N} \to 0$. In general, we assume $\omega(T)h$ tends to a limit $\kappa$, which is not necessarily zero. Then $\omega_0(t) = \omega(t)-\delta(t-T)\kappa$ is a continuous matrix-valued function satisfying the integral equations
\begin{equation}
\begin{aligned}
    -\Sigma_V \omega_0(t)-\frac{\Sigma_W}{2T}\int_0^T s^2\omega_0(s)\,ds+\Sigma_W\int_t^T (s-t)\omega_0(s)\,ds \\ = I+\frac{2t-T}{2}\Sigma_W\kappa, \quad \forall t \in [0, T)
\end{aligned}
\end{equation}
and
\begin{equation}
    \int_0^T \omega_0(s)\,ds = -\kappa.
\end{equation}
Taking the continuous limits of \eqref{Eq: Term2} and \eqref{Eq: Term3}, we see that the probability density $dP'$ is maximized when
\bna \label{Eq: max_posterior}
    &&\nabla\bigg\{-\frac{1}{2}({\bf q}-{\bf a}T)^T\Sigma_W^{-1}({\bf q}-{\bf a}T)-{\bf q}^T\bigg(\frac{1}{4T^2}\int_0^T t^2 \omega(t)\,dt\bigg){\bf q} \\ && +\bigg(\frac{1}{T}\int_0^T \omega(t){\bf y}(t)\,dt\bigg)^T{\bf q}\bigg\} = 0
\ena
where the gradient is taken over the ${\bf a}$ and ${\bf q}$ variables. The best estimates $\bf{\hat{a}}$ and $\bf{\hat{q}}$ are precisely the values that satisfy \eqref{Eq: max_posterior}. Setting the derivative with respect to ${\bf a}$ to 0 yields
\begin{equation}
    {\bf \hat{q}} = {\bf \hat{a}}T,
\end{equation}
and the derivative with respect to ${\bf q}$ equals
\begin{equation}
    -\gamma {\bf \hat{q}}+\beta({\bf y}) = 0,
\end{equation}
where
\begin{equation}
    \gamma = \frac{1}{2T^2}\left(\int_0^T t^2 \omega_0(t)\,dt+T^2\kappa\right)
\end{equation}
and
\begin{equation}
    \beta({\bf y}) = \frac{1}{T}\left(\int_0^T \omega_0(t){\bf y}(t)+\kappa {\bf y}(T) \right).
\end{equation}
We conclude that
\begin{equation}
    {\bf \hat{q}} = \gamma^{-1}\beta({\bf y}).
\end{equation}

\acks{This work was supported by the Air Force Office of Scientific Research,
  award FA9550-19-1-0005.}

\bibliography{references}

\end{document}